\documentclass[oneside,english]{amsart}
\usepackage[T1]{fontenc}
\usepackage{color}
\usepackage{bm}
\usepackage{amstext}
\usepackage{amsthm}
\usepackage{amssymb}
\usepackage{esint}
\usepackage[all]{xy}
\PassOptionsToPackage{normalem}{ulem}
\usepackage{ulem}

\makeatletter
\numberwithin{equation}{section}
\numberwithin{figure}{section}
\theoremstyle{plain}
\newtheorem{thm}{\protect\theoremname}[section]
  \theoremstyle{plain}
  \newtheorem{lem}[thm]{\protect\lemmaname}
  \theoremstyle{remark}
  \newtheorem{rem}[thm]{\protect\remarkname}
  \theoremstyle{plain}
  \newtheorem{prop}[thm]{\protect\propositionname}
  \theoremstyle{definition}
  \newtheorem{defn}[thm]{\protect\definitionname}
  \theoremstyle{plain}
  \newtheorem{conjecture}[thm]{\protect\conjecturename}

\newcounter{myparagraph}[subsection]
\newcommand{\myparagraph}{\refstepcounter{myparagraph}}
\renewcommand{\themyparagraph}{{\arabic{section}.\arabic{subsection}.\alph{myparagraph}}}
\newcommand*{\para}[1]{\vskip0.3cm\noindent\hspace{-3pt}\myparagraph{\bf \themyparagraph.{{\,#1.}}}}

\usepackage[all]{xy}

\usepackage{color}
\usepackage{amsthm}
\usepackage{bm}
\usepackage{graphicx, color, ulem}

\usepackage[labelformat=empty]{caption}

\makeatother

\usepackage{babel}
  \providecommand{\conjecturename}{Conjecture}
  \providecommand{\definitionname}{Definition}
  \providecommand{\lemmaname}{Lemma}
  \providecommand{\propositionname}{Proposition}
  \providecommand{\remarkname}{Remark}
\providecommand{\theoremname}{Theorem}

\begin{document}

\title{K3 surfaces from configurations of six lines in $\mathbb{P}^{2}$
and mirror symmetry II ~\\---~$\lambda_{K3}$-functions ~---}

\author{Shinobu Hosono, Bong Lian and Shing-Tung Yau}
\begin{abstract}
We continue our study on the hypergeometric system $E(3,6)$ which
describes period integrals of the double cover family of K3 surfaces.
Near certain special boundary points in the moduli space of the K3
surfaces, we construct the local solutions and determine the so-called
mirror maps expressing them in terms of genus two theta functions.
These mirror maps are the K3 analogues of the elliptic $\lambda$-function.
We find that there are two non-isomorphic definitions of the lambda
functions corresponding to a flip in the moduli space. We also discuss
mirror symmetry for the double cover K3 surfaces and their higher
dimensional generalizations. A follow up paper will describe more
details of the latter.
\end{abstract}

\maketitle
\tableofcontents{}

\section{\textbf{\textup{Introduction}}}

\vskip0.2cm

Consider elliptic curves given as double covers over $\mathbb{P}^{1}$
branched along four points in general positions. These curves define
a family of elliptic curves over the configuration space $\mathcal{M}_{4}$
of four points in $\mathbb{P}^{1}$, which is called Legendre family.
The elliptic lambda function is a modular function associated to this
family. This gives the uniformization of the period map defined as
a multi-valued function from $\mathcal{M}_{4}$ to the upper-half
plane $\mathbb{H}_{+}$. In this paper we will define a generalization
of this elliptic lambda function for a certain family of K3 surfaces.

We will consider double covers of $\mathbb{P}^{2}$ branched along
six lines in general positions which are singular at fifteen intersection
points of the lines. Blowing-up at the singularities gives smooth
K3 surfaces over the configuration space $\mathcal{M}_{6}$ of six
lines, which we called \textit{double cover family of K3 surfaces
}in the previous work \cite{HLTYpartI}. This family has been studied
in many contexts (see \cite{YoshidaBook} for example) as a natural
generalization of the Legendre family over $\mathcal{M}_{4}$. In
particular, in \cite{YoshidaEtal}, monodromy property of the period
map has been determined completely. However since the moduli space
$\mathcal{M}_{6}$ is singular, we need to find suitable resolutions
to study throughly the analytic properties of the period maps. In
\cite{HLTYpartI}, we have found nice resolutions $\widetilde{\mathcal{M}}_{6}$
and $\widetilde{\mathcal{M}}_{6}^{+}$ from the viewpoint of mirror
symmetry and Picard-Fuchs differential equations of period integrals.
The aim of this paper is to define K3 analogues to the elliptic lambda
function based on these resolutions. 

Let us recall that, for the definition of the the elliptic lambda
function, the hypergeometric series 

\begin{equation}
\omega_{0}(z)=\sum_{n\geq0}\frac{1}{\Gamma(\frac{1}{2})^{2}}\frac{\Gamma(n+\frac{1}{2})^{2}}{\Gamma(n+1)^{2}}z^{n}\label{eq:w0-elliptic}
\end{equation}
and the differential equation (Picard-Fuchs equation) satisfied by
it plays a central role. In this case, Picard-Fuchs differential equation
is given by Gauss's hypergeometric differential equation, and its
solutions determine the period integrals of the Legendre family. The
period map $\mathcal{P}:\mathcal{M}_{4}\to\mathbb{H}_{+}$ is basically
given by the ratio of the solutions with the monodromy group the congruence
subgroup $\Gamma(2)$ of $\Gamma=PSL(2,\mathbb{Z})$. The elliptic
lambda function is the inverse map $\mathbb{H}_{+}/\Gamma(2)\to\mathcal{M}_{4}$
with suitable boundary properties near the cusps. The following explicit
forms for the lambda function and the hypergeometric series are well-known:
\begin{equation}
\lambda(\tau)=\frac{\vartheta_{2}(\tau)^{4}}{\vartheta_{3}(\tau)^{4}},\;\;\omega_{0}(\lambda(\tau))^{2}=\vartheta_{3}(\tau)^{4}\label{eq:elliptic-lambda}
\end{equation}
where $\tau\in\mathbb{H}_{+}$. See Section \ref{para:theta-ell-def}
for the definitions of theta functions.

The generalization to a family of K3 surfaces has been studied extensively
in the '90s \cite{YoshidaEtal,Matsu93,YoshidaBook}. However, it was
not clear how to resolve the moduli space $\mathcal{M}_{6}$ to construct
analogues of the expressions (\ref{eq:elliptic-lambda}). In \cite{HLTYpartI},
we have found natural resolutions $\widetilde{\mathcal{M}}_{6}$ and
$\widetilde{\mathcal{M}}_{6}^{+}$ of $\mathcal{M}_{6}$ which are
related by a four dimensional flip. In this paper, corresponding to
these resolutions, we will construct two definitions for K3 analogues
of the elliptic lambda function; they differ in their behaviors near
the exceptional divisors of the resolutions. We call these analogues
\textit{K3 lambda functions} $\lambda_{k}$ and $\lambda_{k}^{+}$,
respectively. These might be called \textit{K3 lambda maps} precisely,
but we continue to use the word ``function'' to indicate the generalization
of elliptic lambda function. 

The K3 lambda functions are naturally identified with the so-called
mirror map \cite{HKTY,HLY} for the family of K3 surfaces. In connection
to this, we will also discuss mirror symmetry of the family; we will
find that the mirror geometry is a (singular) K3 surface which is
given as a double cover of a del Pezzo surface $Bl_{3}\mathbb{P}^{2}$,
a three point blow-up of $\mathbb{P}^{2}$. 

Below we summarize the K3 lambda functions and hypergeometric series
which we shall formulate in this paper. 

~

\noindent$\bullet$ \uline{K3 lambda function $\lambda_{k}$ }:
The mirror map is given by $z_{k}=\lambda_{k}$ with

\begin{equation}
\begin{alignedat}{5}\lambda_{1} & = &  & \frac{\Theta_{3}^{2}+\Theta_{9}^{2}-\omega_{0}^{2}}{\omega_{0}^{2}-\Theta_{7}^{2}} & ,\quad & \lambda_{2} & = &  & \frac{\Theta_{3}^{2}+\Theta_{9}^{2}-\omega_{0}^{2}}{\omega_{0}^{2}-\Theta_{9}^{2}} & ,\\
\lambda_{3} & = &  & \frac{(\omega_{0}^{2}-\Theta_{7}^{2})(\omega_{0}^{2}-\Theta_{9}^{2})}{\omega_{0}^{2}(\Theta_{4}^{2}+\Theta_{9}^{2}-\omega_{0}^{2})} & ,\quad & \lambda_{4} & = &  & \frac{\Theta_{4}^{2}+\Theta_{9}^{2}-\omega_{0}^{2}}{\Theta_{3}^{2}+\Theta_{9}^{2}-\omega_{0}^{2}}
\end{alignedat}
\label{eq:intro-lambdaK3}
\end{equation}

\[
\omega_{0}(\lambda_{1},\lambda_{2},\lambda_{3},\lambda_{4})^{2}=\frac{1}{2\Theta_{8}^{2}}\left\{ \Theta_{7}^{2}\Theta_{8}^{2}-\Theta_{10}^{2}\Theta_{5}^{2}+\Theta_{6}^{2}\Theta_{9}^{2}-\widetilde{\Theta}\right\} 
\]
where 
\[
\omega_{0}(z)=\sum_{n_{1},n_{2},n_{3},n_{4}\geq0}c(n_{1},n_{2},n_{3},n_{4})z_{1}^{n_{1}}z_{2}^{n_{2}}z_{3}^{n_{3}}z_{4}^{n_{4}}
\]
with $c(n)=c(n_{1},n_{2},n_{3},n_{4})$ given by 
\[
c(n):=\frac{1}{\Gamma(\frac{1}{2})^{3}}\frac{\Gamma(n_{1}+\frac{1}{2})\Gamma(n_{2}+\frac{1}{2})\Gamma(n_{3}+\frac{1}{2})}{\Pi_{_{i=1}}^{3}\Gamma(n_{4}-n_{i}+1)\cdot\Pi_{1\leq j<k\leq3}\Gamma(n_{j}+n_{k}-n_{4}+1)},
\]
and 

\[
\widetilde{\Theta}=\frac{2^{2}}{3\cdot5}\Theta=-64(q_{3}-q_{4})\left\{ \frac{q_{1}q_{2}(1-q_{3}q_{4})}{q_{3}q_{4}}+\cdots\right\} 
\]
is the weight four theta function, see Appendix \ref{sec:App-Theta-fn}.

\noindent$\bullet$ \uline{K3 lambda function $\lambda_{k}^{+}$
}: The mirror map is given by $\tilde{z}_{k}=\lambda_{k}^{+}$ with

\begin{equation}
\begin{alignedat}{4}\lambda_{1}^{+} & = &  & \frac{\Theta_{3}^{2}+\Theta_{9}^{2}-\omega_{0}^{2}}{\omega_{0}^{2}-\Theta_{6}^{2}} & ,\quad & \lambda_{2}^{+} & = & \frac{\Theta_{4}^{2}+\Theta_{9}^{2}-\omega_{0}^{2}}{\omega_{0}^{2}-\Theta_{6}^{2}},\\
\lambda_{3}^{+} & = &  & \frac{\omega_{0}^{2}-\Theta_{6}^{2}}{\omega_{0}^{2}-\Theta_{9}^{2}} & ,\quad & \lambda_{4}^{+} & = & \frac{(\omega_{0}^{2}-\Theta_{6}^{2})^{2}(\omega_{0}^{2}-\Theta_{9}^{2})}{\omega_{0}^{2}(\Theta_{3}^{2}+\Theta_{9}^{2}-\omega_{0}^{2})(\Theta_{4}^{2}+\Theta_{9}^{2}-\omega_{0}^{2})}
\end{alignedat}
\label{eq:intro-lambdaK3plus}
\end{equation}
\[
\omega_{0}(\lambda_{1}^{+},\lambda_{2}^{+},\lambda_{3}^{+},\lambda_{4}^{+})^{2}=\frac{1}{2\Theta_{8}^{2}}\left\{ \Theta_{7}^{2}\Theta_{8}^{2}-\Theta_{10}^{2}\Theta_{5}^{2}+\Theta_{6}^{2}\Theta_{9}^{2}-\widetilde{\Theta}\right\} 
\]
where 
\[
\omega_{0}(\tilde{z})=\sum_{n_{1},n_{2},n_{3},n_{4}\geq0}\tilde{c}(n_{1},n_{2},n_{3},n_{4})\tilde{z}_{1}^{n_{1}}\tilde{z}_{2}^{n_{2}}\tilde{z}_{3}^{n_{3}}\tilde{z}_{4}^{n_{4}}
\]
with $\tilde{c}(n)=\tilde{c}(n_{1},n_{2},n_{3},n_{4})$ given by 
\[
\tilde{c}(n):=\frac{1}{\Gamma(\frac{1}{2})^{3}}\frac{\Gamma(n_{1}+n_{2}-n_{3}-n_{4}+\frac{1}{2})\Gamma(n_{3}+\frac{1}{2})\Gamma(n_{4}+\frac{1}{2})}{\Pi_{i=1,2}\Pi_{j=3,4}\Gamma(n_{i}-n_{j}+1)\cdot\Pi_{i=1,2}\Gamma(n_{3}+n_{4}-n_{i}+1)}.
\]

~

As is the case for the relation $\omega_{0}(\lambda(\tau))^{2}=\vartheta_{3}(\tau)^{4}$,
the above equalities for $\omega_{0}(\lambda_{1},...,\lambda_{4})^{2}$
and $\omega_{0}(\lambda_{1}^{+},...,\lambda_{4}^{+})^{2}$ are local
expressions, which will be multiplied by suitable weight factors under
the monodromy transformations (or, equivalently, under the modular
transformations). However the forms of lambda functions $\lambda_{k}$
and $\lambda_{k}^{+}$ given above are global functions defined over
the resolutions $\widetilde{\mathcal{M}}_{6}$ and $\widetilde{\mathcal{M}}_{6}^{+}$,
respectively. 

The construction of this paper is as follows. In Section \ref{sec:elliptic-lambda},
we will describe the Legendre family in a form which generalizes to
the double cover family of K3 surfaces. In particular, we describe
in detail the well-known action on $\mathcal{M}_{4}\simeq\mathbb{P}^{1}$
of the symmetric group $S_{3}$. Based on the commutative diagram
(\ref{eq:E24diagram}), which is equivariant under $S_{3}\simeq\Gamma/\Gamma(2)$,
we shall characterize the lambda function and the property of the
hypergeometric series (\ref{eq:elliptic-lambda}). In Section \ref{sec:Master-Equation},
we summarize known-results about the double cover family of K3 surfaces
including the results in our previous work \cite{HLTYpartI}. We will
then use them to formulate a \textit{master equation} for our definition
of the lambda functions. In Section \ref{sec:Gen-Frobenius-method},
we summarize the generalized Frobenius method \cite{HKTY,HLY} which
describes the local solutions near certain special boundary points
called \textit{large complex structure limit points} (LCSLs). We also
present an explicit form of period integrals of the family which is
valid near the LCSLs. In Section \ref{sec:Solving-Master-Eqs}. we
will describe the period map using local solutions near the special
boundary points. We will find a consistent form of the master equation
with the local expression of the period maps. We then solve the master
equation algebraically to obtain the K3 analogues of the lambda function.
In Section \ref{sec:Mirror-symmetr}, we will discuss the mirror geometry
of the double cover family of K3 surfaces. Each section relies on
previous results scattered in many works. We will present these in
appendices. In Appendix \ref{sec:App-representation-S6}, we present
some explicit formulas for the representation $S_{6}\to\mathrm{Aut}(\widetilde{\mathcal{M}}_{6})$
which is a generalization of the well-known representation $S_{3}\to\mathrm{Aut}(\mathcal{M}_{4})\simeq\mathrm{Aut}(\mathbb{P}^{1})$.
This is a byproduct of our arguments, but should be of some interest
in its own right. 

\vskip0.3cm

\textbf{Acknowledgements:} S.H. would like to thank for the warm hospitality
at the CMSA at Harvard University where progress was made. S.H. is
supported in part by Grant-in Aid Scientific Research (C 16K05105,
S 17H06127, A 18H03668 S.H.). B.H.L and S.-T. Yau are supported by
the Simons Collaboration Grant on Homological Mirror Symmetry and
Applications 2015--2019.

~

~

\newpage

\section{\textbf{\textup{The elliptic $\lambda$-function \label{sec:elliptic-lambda}}}}

\subsection{Legendre family }

The double cover family of K3 surfaces shares many properties with
the corresponding family of elliptic curves, i.e. the Legendre family.
It is helpful to summarize the well-known results of the Legendre
family in the forms which generalize to the double cover family of
K3 surfaces. 

\para{The configuration space of four points in $\mathbb{P}^1$} The
Legendre family is a family of elliptic curves given as double covers
of $\mathbb{P}^{1}$ branched at four points in general position.
To describe the family, let us introduce a data given by 
\[
A=\left(\begin{matrix}a_{01} & a_{02} & a_{03} & a_{04}\\
a_{11} & a_{02} & a_{13} & a_{14}
\end{matrix}\right)\in M_{2,4},
\]
where $M_{2,4}$ is the set of $2\times4$ complex matrices. We denote
its open dense subset by 
\[
M_{2,4}^{o}=\left\{ A\in M_{2,4}\mid[i_{1}\,i_{2}]\not=0\;(1\leq i_{1},i_{2}\leq4)\right\} 
\]
with $[i_{1}\,i_{2}]=\left|\begin{smallmatrix}a_{0i_{1}} & a_{0i_{2}}\\
a_{1i_{1}} & a_{1i_{2}}
\end{smallmatrix}\right|$. For $A\in M_{2,4}^{o}$, we consider an elliptic curve branched
at four points specified by $A$: 
\[
{\tt y}^{2}=\prod_{i=1}^{4}(a_{0i}{\tt x_{0}}+a_{1i}{\tt x}_{1}).
\]
Isomorphism classes of these elliptic curves are parametrized by the
quotient space $GL(2,\mathbb{C})\diagdown M_{2,4}^{o}\diagup(\mathbb{C}^{*})^{4}$.
This quotient is naturally compactified by the GIT quotient \cite{DoOrt,Reuv}
which is called the configuration space $\mathcal{M}_{4}$ of four
points on $\mathbb{P}^{1}$. 

It is easy to see the isomorphism $\mathcal{M}_{4}\simeq\mathbb{P}^{1}$.
In fact, in the quotient, any matrix $A\in M_{2,4}^{o}$ can be transformed
into the form $\left(\begin{smallmatrix}1 & 0 & 1 & 1\\
0 & 1 & 1 & z
\end{smallmatrix}\right)$ with 
\[
z=\frac{[2\,3][1\,4]}{[1\,3][2\,4]},
\]
which can be identified with the cross ratio of four points. 

\para{Perid integrals and Picard-Fuchs equation} The period integrals
over cycles in $H_{2}(X,\mathbb{Z})$ are given by 
\begin{equation}
\bar{\omega}_{C}(a)=\int_{C}\frac{d\mu}{\sqrt{\prod_{i=1}^{4}(a_{0i}\mathtt{x}_{0}+a_{1i}\mathtt{x}_{1})}}\qquad(d\mu=i_{E}d\mathtt{x}_{0}\wedge d\mathtt{x}_{1},\,\,C\in H_{2}(X,\mathbb{Z}),\label{eq:periodE24-bar}
\end{equation}
where $i_{E}$ is the contraction with the Euler vector field $E=\mathtt{x}_{0}\frac{\partial\;}{\partial\mathtt{x}_{0}}+\mathtt{x}_{1}\frac{\partial\;}{\partial\mathtt{x}_{1}}.$
They are solutions to the Picard-Fuchs equation, which is given by
the hypergeometric system $E(2,4)$, i.e. the hypergometric system
on Grassmannian $G(2,4)$ \cite{Ao,GelfandGraev}. The hypergeometric
system $E(2,4)$ reduces locally to the so-called GKZ (Gel'fand-Kapranov-Zelevinski)
system \cite{GKZ1} when we represent an equivalence class $[A]\in GL(2,\mathbb{C})\diagdown M_{2,4}^{o}/(\mathbb{C}^{*})^{4}$
by 
\begin{equation}
A=\left(\begin{matrix}1 & 0 & a_{1} & b_{0}\\
0 & 1 & a_{0} & b_{1}
\end{matrix}\right).\label{eq:A-E24-E2X}
\end{equation}
This reduces the $GL(2,\mathbb{C})\times(\mathbb{C}^{*})^{4}$ action
on $M_{2,4}^{o}$ to the torus actions of the form $(\mathbb{C}^{*})^{2}\diagdown M_{2,4}^{o}\diagup(\mathbb{C}^{*})^{4}$
which preserve the above form of the matrix $A$, i.e., 
\begin{equation}
T=\left\{ (g,t)\in GL(2,\mathbb{C})\times(\mathbb{C}^{*})^{4}\mid g\left(\begin{matrix}1\,\,0\,*\,*\\
0\,\,1\,*\,*
\end{matrix}\right)t=\left(\begin{matrix}1\,\,0\,*\,*\\
0\,\,1\,*\,*
\end{matrix}\right)\right\} \diagup\sim,\label{eq:T-group-M4}
\end{equation}
where $(g,t)\sim(\lambda g,\lambda^{-1}t)\;(\lambda\in\mathbb{C}^{*})$(see
\cite[Sect.2.4]{HLTYpartI} for more details). The GKZ system is described
by the affine parameters $(a_{0,}b_{0},a_{1},b_{1})\in\mathbb{C}^{4}$,
and is defined on a natural toric compactification $\mathcal{M}_{SecP}$
of the parameter space. Following Sect. 3 of \cite{HLTYpartI}, it
is easy to see $\mathcal{M}_{4}\simeq\mathcal{M}_{SecP}\simeq\mathbb{P}^{1}$.
In particular, we arrive at the cross ratio 
\begin{equation}
z=\frac{[2\,3][1\,4]}{[1\,3][2\,4]}=\frac{a_{1}b_{1}}{a_{0}b_{0}},\label{eq:cross-ratio}
\end{equation}
as an affine coordinate of $\mathcal{M}_{SecP}$. We write this coordinate
as a monomial $z=\mathtt{a}^{\ell}$ by introducing $\mathtt{a}=(-a_{0},-b_{0},a_{1},b_{1})$
and $\ell=(-1,-1,1,1)$. After scaling $\bar{\omega}_{C}(a)$ by the
factor $(a_{0}b_{0})^{\frac{1}{2}}$, it is easy to see that the period
integral 
\begin{equation}
\omega(z)=\int_{C}\frac{\sqrt{a_{0}b_{0}}}{\sqrt{(a_{0}+a_{1}\frac{1}{\mathtt{x}_{1}})(b_{0}+b_{1}\mathtt{x}_{1}})}\frac{d\mathtt{x_{1}}}{\mathtt{x}_{1}}\label{eq:periodE24-normalized}
\end{equation}
satisfies the following differential equation, Picard-Fuchs equation,
\begin{equation}
\mathcal{D}_{z}\omega(z)=\big\{\theta_{z}^{2}+z(\theta_{z}+\frac{1}{2})^{2}\big\}\omega(z)=0,\label{eq:PF-ellip}
\end{equation}
with $\theta_{z}:=z\frac{d\;}{dz}$ (cf. \cite[Sect.3]{HLTYpartI}).
This differential equation has three regular singularities at $\left\{ 0,1,\infty\right\} $,
and the local solutions around $z=0$ are generated by the standard
Frobenius method;
\begin{equation}
\omega_{0}(z)=\omega(z,\rho)\vert_{\rho=0},\;\;\quad\omega_{1}(z)=\frac{2}{2\pi i}\frac{\partial\;}{\partial\rho}\omega(z,\rho)\vert_{\rho=0}\label{eq:w0-w1-E24}
\end{equation}
where $\omega(z,\rho):=\sum_{n\geq0}c(n+\rho)z^{n+\rho}$ with $c(n)=\frac{1}{\Gamma(\frac{1}{2})^{2}}\frac{\Gamma(n+\frac{1}{2})^{2}}{\Gamma(n+1)}$.
Here the constant factors $\frac{2}{2\pi i}$ and $\frac{1}{\Gamma(\frac{1}{2})^{2}}$
are fixed to have integral monodromies for the analytic continuations
of the solutions $\,^{t}(\omega_{1}(z),\omega_{0}(z))$ over $\mathbb{P}^{1}\setminus\left\{ 0,1,\infty\right\} $.
The ratio of the period integral $\tau:=\frac{\omega_{1}(z)}{\omega_{0}(z)}$
defines the multi-valued period map $\mathcal{P}:\mathcal{M}_{4}\to\mathbb{H}_{+}$,
where $\mathbb{H}_{+}$ is the upper half plane. The inverse of the
period map $z=z(\tau)$ is one of the simplest example of the so-called
mirror map. In the present case, this mirror map $z(\tau)$ coincides
with the elliptic lambda function $\lambda(\tau)$ which is a modular
function on the level two subgroup $\Gamma(2)$ of $\Gamma:=PSL(2,\mathbb{Z})$.

\subsection{Theta functions and semi-invariants }

Using the local solutions of (\ref{eq:PF-ellip}), we can describe
the mirror map locally, for example, in terms of the $q$-expansion
with $q:=e^{\pi i\tau}$. For global properties, we use modular forms
on $\Gamma(2)$, whose ring of even weights are known to be generated
by classical theta functions $\theta_{2}(\tau)^{4},\theta_{3}(\tau)^{4}$
and $\theta_{4}(\tau)^{4}$. It is useful to summarize the relation
to the period map in the following diagram:
\begin{equation}
\xymatrix{\mathcal{M}_{4}\;\;\ar[dr]_{\Phi_{Y}}\ar[rr]^{\mathcal{P}} &  & \;\;\mathbb{H}_{+}\ar[dl]^{\Phi}\\
 & \mathbb{P}^{2},
}
{\color{red}}\label{eq:E24diagram}
\end{equation}
where $\mathcal{P}$ is the period map and $\Phi(\tau):=[\theta_{2}(\tau)^{4},\theta_{3}(\tau)^{4},\theta_{4}(\tau)^{4}]$.
The map $\Phi_{Y}:\mathcal{M}_{4}\to\mathbb{P}^{2}$ is defined by
semi-invariants of the GIT quotient which we describe in detail below. 

\para{Theta functions} \label{para:theta-ell-def}We follow the standard
definition of the theta functions: $\theta_{2}(\tau)=\sum_{n\in\mathbb{Z}}q^{(n+\frac{1}{2})^{2}}$,
$\theta_{3}(\tau)=\sum_{n\in\mathbb{Z}}q^{n^{2}}$ and $\theta_{4}(\tau)=\sum_{n\in\mathbb{Z}}(-1)^{n}q^{n^{2}}$
which satisfy one linear relation $\theta_{2}(\tau)^{4}+\theta_{4}(\tau)^{4}-\theta_{3}(\tau)^{4}=0$.
To a parallel formula with the K3 case, we associate the theta functions
to certain partitions as follows:
\[
\Theta\left(\begin{matrix}1\,2\\
3\,4
\end{matrix}\right)(\tau)=\theta_{4}(\tau)^{2},\;\,\Theta\left(\begin{matrix}1\,3\\
2\,4
\end{matrix}\right)(\tau)=\theta_{3}(\tau)^{2},\;\,\Theta\left(\begin{matrix}1\,4\\
2\,3
\end{matrix}\right)(\tau)=\theta_{2}(\tau)^{2}.
\]
They have the (anti-)symmetry properties $\Theta\left(\begin{matrix}i\,j\\
k\,l
\end{matrix}\right)=\Theta\left(\begin{matrix}k\,l\\
i\,j
\end{matrix}\right)$ and 
\[
\Theta\left(\begin{matrix}m\,n\\
r\,s
\end{matrix}\right)^{2}=\mathrm{sgn}\left(\begin{matrix}m\,n\\
i\,j
\end{matrix}\right)\mathrm{sgn}\left(\begin{matrix}r\,s\\
k\,l
\end{matrix}\right)\Theta\left(\begin{matrix}i\,j\\
k\,l
\end{matrix}\right)^{2}.
\]
Using these, the linear relation $\theta_{2}(\tau)^{4}+\theta_{4}(\tau)^{4}-\theta_{3}(\tau)^{4}=0$
becomes 
\begin{equation}
\Theta\left(\begin{matrix}1\,2\\
3\,4
\end{matrix}\right)^{2}-\Theta\left(\begin{matrix}1\,3\\
2\,4
\end{matrix}\right)^{2}+\Theta\left(\begin{matrix}1\,4\\
2\,3
\end{matrix}\right)^{2}=0.\label{eq:theta-relation}
\end{equation}

\para{Semi-invariants} According to geometric invariant theory \cite{DoOrt},
the map $\Phi_{Y}:\mathcal{M}_{4}\to\mathbb{P}^{2}$ is defined by
the ring generators of semi-invariants of the $GL(2,\mathbb{C})\times(\mathbb{C}^{*})^{4}$
actions on $M_{2,4}$. Concretely, it is given by $\Phi_{Y}([A])=[Y_{0},Y_{1},Y_{2}]$
with 
\[
Y_{0}=[1\,2][3\,4],\;\;Y_{1}=[1\,3][2\,4],\;\;Y_{2}=[1\,4][2\,3],
\]
where $[i\,j]$ represent the $2\times2$ minors of $A$. These $Y_{k}$'s
satisfy the Pl\"ucker relation $Y_{0}-Y_{1}+Y_{2}=0$ which corresponds
to (\ref{eq:theta-relation}), and the period map $\mathcal{P}$ makes
the diagram (\ref{eq:E24diagram}) commute. 

\para{Affine coordinates from the level two structure} \label{para:affine-coordinate-ell}
The period map $\mathcal{P}:\mathcal{M}_{4}\to\mathbb{H}_{+}$ is
in fact a multi-valued map with its monodromy group $\Gamma(2)$ giving
the isomorphism $\mathcal{M}_{4}\simeq\overline{\Gamma(2)\diagdown\mathbb{H}_{+}}$.
The symmetric group of order three $S_{3}\simeq\Gamma/\Gamma(2)$
acts naturally on $\overline{\Gamma(2)\diagdown\mathbb{H}_{+}}$ as
its aoutomorphisms. These come from the right actions of $S_{4}$
on $\mathcal{M}_{4}$ by $4\times4$ permutation matrices, which induce
the following actions on the cross ratio (\ref{eq:cross-ratio});
\[
z\mapsto z^{\sigma}=\varphi_{\sigma}(z)=\frac{[\sigma(2)\sigma(3)][\sigma(1)\sigma(4)]}{[\sigma(1)\sigma(3)][\sigma(2)\sigma(4)]}\;(\sigma\in S_{4}).
\]
Because of the non-trivial isotropy group $H$, the $S_{4}$ group
action actually reduces to the factor group $S_{3}\simeq S_{4}/H$.
We will identify this factor group with the subgroup $S_{3}=\left\{ \sigma\in S_{4}\mid\sigma(4)=4\right\} $.
The explicit forms of the automorphisms $\varphi_{\sigma}:\mathcal{M}_{4}\to\mathcal{M}_{4}$
are summarized in the following table:

\begin{equation}
\begin{matrix}\sigma & : & e & (12) & (23) & (23)(12) & (12)(23) & (13)\\
z^{\sigma} & : & z & \frac{1}{z} & \frac{z}{z-1} & 1-\frac{1}{z} & \frac{1}{1-z} & 1-z
\end{matrix}\label{eq:table-zSigma}
\end{equation}

In what follows, we shall read the above automorphisms $z^{\sigma}=\varphi_{\sigma}(z)$
as the coordinate transformations between different affine charts
which cover $\mathcal{M}_{4}\simeq\mathbb{P}^{1}$. 
\begin{lem}
\label{lem:AsigmaB}Let $p\in\mathcal{M}_{4}$ be any point represented
by $2\times4$ matrix $A$. Then the following properties hold: \end{lem}
\begin{enumerate}
\item There is a right action by $\sigma\in S_{4}$ which brings $A$ into
$A\sigma$ of the form: 
\begin{equation}
A\sigma=B(\sigma)\left(\begin{matrix}1\,\,\,0\,\\
0\,\,\,1\,
\end{matrix}\;\begin{matrix}a_{1}^{\sigma}\,\,b_{0}^{\sigma}\\
a_{0}^{\sigma}\,\,b_{1}^{\sigma}
\end{matrix}\right),\,a_{0}^{\sigma}b_{0}^{\sigma}\not=0,\label{eq:AsigmaB}
\end{equation}
where $B(\sigma)$ is a $2\times2$ regular matrix.
\item When we change the representative of $p=[A]$ to $gAt$ by $(g,t)\in T$,
the same $\sigma$ brings $gAt$ to the form (\ref{eq:AsigmaB}) with
$gB(\sigma)h^{-1}$, where $h$ is determined uniquely by the condition
$(h,\sigma^{-1}t\sigma)\in T$ (see (\ref{eq:T-group-M4}) for the
definition of $T$). \end{enumerate}
\begin{proof}
(1) The moduli space $\mathcal{M}_{4}$ parametrizes the equivalence
classes of semi-stable configurations of four points in $\mathbb{P}^{1}$.
The claim follows from the fact that no three points coincide for
a semi-stable configuration represented by $A$. (2) Suppose $A\sigma$
has the form (\ref{eq:AsigmaB}). Then we have 
\[
\begin{matrix}\begin{aligned}gAt\,\sigma=gA\sigma\,(\sigma^{-1}t\sigma) & = & gB_{2}(\sigma)\left(\begin{matrix}1\,\,\,0\,\\
0\,\,\,1\,
\end{matrix}\;\begin{matrix}a_{1}^{\sigma}\,\,b_{0}^{\sigma}\\
a_{0}^{\sigma}\,\,b_{1}^{\sigma}
\end{matrix}\right)(\sigma^{-1}t\sigma)\qquad\quad\\
 & = & gB_{2}(\sigma)h^{-1}\cdot h\left(\begin{matrix}1\,\,\,0\,\\
0\,\,\,1\,
\end{matrix}\;\begin{matrix}a_{1}^{\sigma}\,\,b_{0}^{\sigma}\\
a_{0}^{\sigma}\,\,b_{1}^{\sigma}
\end{matrix}\right)(\sigma^{-1}t\sigma),
\end{aligned}
\end{matrix}
\]
where $h$ is unique by the condition $(h,\sigma^{-1}t\sigma)\in T$
with $T\simeq\mathbb{C}^{*}$ given in (\ref{eq:T-group-M4}). Since
$(h,\sigma^{-1}t\sigma)\in T$ acts on the matrix entries by $\mathbb{C}^{*}$
actions, the condition $a_{0}^{\sigma}b_{0}^{\sigma}\not=0$ is retained.
\end{proof}
Let us introduce the following notation for $A=(a_{ij})$; 
\[
z(A):=\frac{[23][14]}{[13][24]},\;\;z^{\sigma}(A):=z(A\sigma)=\frac{[\sigma(2)\sigma(3)][\sigma(1)\sigma(4)]}{[\sigma(1)\sigma(3)][\sigma(2)\sigma(4)]}.
\]
Based on Lemma \ref{lem:AsigmaB}, we define for $\sigma\in S_{4}$
the subset of $\mathcal{M}_{4}$ by 
\begin{equation}
M_{\sigma}:=\left\{ [A]\in\mathcal{M}_{4}\mid A\sigma\text{ has the form (\ref{eq:AsigmaB})}\right\} .\label{eq:Czsigma-M4}
\end{equation}
Then we have $z^{\sigma}(A)=\frac{a_{1}^{\sigma}b_{1}^{\sigma}}{a_{0}^{\sigma}b_{0}^{\sigma}}$
for $[A]\in M_{\sigma}$. This shows that $M_{\sigma}\simeq\mathbb{C}$
and $z^{\sigma}$ is an affine coordinate on it. We will denote by
$\mathbb{C}_{z^{\sigma}}$ this affine open set $M_{\sigma}$ with
its coordinate function $z^{\sigma}$. Now, it is easy to see that
we have the covering of $\mathcal{M}_{4}$ by these affine open sets:
\begin{equation}
\mathcal{M}_{4}=\bigcup_{\sigma\in S_{3}}\mathbb{C}_{z^{\sigma}}.\label{eq:M4-sigma-Union}
\end{equation}

When we have $z=z(A)$ for a configuration $[A]\in\mathbb{C}_{z}\cap\mathbb{C}_{z^{\sigma}}$,
the coordinate function $z^{\sigma}$ of $\mathbb{C}_{z^{\sigma}}$
evaluates the same point by $z^{\sigma}(A)=z(A\sigma)$. By definition,
these two values are related by $z^{\sigma}(A)=\varphi_{\sigma}(z(A))$. 
\begin{rem}
$\varphi_{\sigma}$'s are anti-homomorphisms, $\varphi_{\sigma\tau}=\varphi_{\tau}\circ\varphi_{\sigma}$,
since $z^{\sigma\tau}(A)=z^{\tau}(A\sigma)$. 
\end{rem}

\subsection{Transformation properties of semi-invariants}

Let us recall that the semi-invariants $Y_{k}=Y_{k}(A)$ are homogeneous
polynomials of matrix elements of $A$. We will express these semi-invariants
as some polynomials in the affine coordinate of $\mathbb{C}_{z}$,
and describe the transformation properties of these polynomials under
the coordinate changes $z^{\sigma}=\varphi_{\sigma}(z)$. This simply
reproduces the well-known properties of the elliptic lambda function
for the Legendre family. However, this will become our guiding principle
to define the K3 analogues of the elliptic lambda functions. 

\para{Polynomials $P_I$} \label{para:elliptic-def-P}It is convenient
to write $Y_{k}(A)\,(k=0,1,2)$ as 
\[
Y_{I}(A)=[i\,j][k\,l]
\]
introducing the ordered set $I=\{\{i,j\},\{k,l\}\}$. Assume $A$
has a special form $A_{0}=\left(E_{2}\,X\right)=\left(\begin{smallmatrix}1 & 0\\
0 & 1
\end{smallmatrix}\begin{smallmatrix}a_{1} & b_{0}\\
a_{0} & b_{1}
\end{smallmatrix}\right)$ with $a_{0}b_{0}\not=0$. For such $A_{0}$, we define 
\[
P_{I}:=\frac{1}{a_{0}b_{0}}Y_{I}(A_{0}).
\]
It is easy to verify that $P_{I}$'s are polynomials of $z=z(A_{0})=\frac{a_{1}b_{1}}{a_{0}b_{0}}$;
and they are given by 
\begin{equation}
P_{I}(z)=\;z-1,\;\;-1,\;\;-z,\label{eq:tildeYI-E24}
\end{equation}
for $I=\{\{1,2\},\{3,4\}\},\{\{1,3\},\{2,4\}\}$ and $\{\{1,4\},\{2,3\}\}$,
respectively. 

\para{Semi-invariants in affine coordinates} We can express the semi-invariants
$Y_{I}(A)$ for general $A$ in terms of the polynomial $P_{I}$ given
in (\ref{eq:tildeYI-E24}). Let us first note that, by definition,
we have the following relation for $A=(\bm{a}_{1}\,\bm{a}_{2}\,\bm{a}_{3}\,\bm{a}_{4})$:
\begin{equation}
Y_{I}(A\sigma)=\det(\bm{a}_{\sigma(i)}\,\bm{a}_{\sigma(j)})\,\det(\bm{a}_{\sigma(k)}\,\bm{a}_{\sigma(l)})=Y_{\sigma(I)}(A),\label{eq:Asigam-Y}
\end{equation}
where $\sigma(I)=\{\{\sigma(i),\sigma(j)\},\{\sigma(k),\sigma(l)\}\}$. 
\begin{prop}
\label{prop:Y-P-rel-E24}For $A\in M_{2,4}$ such that $[A]\in\mathbb{C}_{z}\cap\mathbb{C}_{z^{\sigma}}$,
we have 
\begin{equation}
\begin{aligned}Y_{I}(A) & =(\det B_{2}(e))^{2}a_{0}^{e}b_{0}^{e}\cdot P_{I}(z(A))\\
 & =(\det B_{2}(\sigma))^{2}a_{0}^{\sigma}b_{0}^{\sigma}\cdot P_{\sigma^{-1}(I)}(z^{\sigma}(A)).
\end{aligned}
\label{eq:YA-by-affineY}
\end{equation}
\end{prop}
\begin{proof}
By the definition of $B_{2}(e)$, we have $A=Ae=B_{2}(e)\,A_{0}$,
from which we obtain $Y_{I}(A)=(\det B_{2}(e))^{2}\,Y_{I}(A_{0})$.
Now, the first equality of (\ref{eq:YA-by-affineY}) follows from
the definition $P_{I}(z(A_{0}))=\frac{1}{a_{0}b_{0}}Y_{I}(A_{0})$.
For the second equality, we use (\ref{eq:Asigam-Y}) to have 
\[
Y_{I}(A)=Y_{\sigma^{-1}(I)}(A\sigma)=(\det B_{2}(\sigma))^{2}\cdot Y_{\sigma^{-1}(I)}((E_{2}\,X_{\sigma})),
\]
where $X_{\sigma}:=\left(\begin{matrix}a_{1}^{\sigma}\,\,b_{0}^{\sigma}\\
a_{0}^{\sigma}\,\,b_{1}^{\sigma}
\end{matrix}\right)$. Noting that 
\[
z^{\sigma}(A)=z(A\sigma)=z((E_{2}\,X_{\sigma}))=\frac{a_{1}^{\sigma}b_{1}^{\sigma}}{a_{0}^{\sigma}b_{0}^{\sigma}},
\]
and $Y_{I}(A_{0})=a_{0}b_{0}P_{I}(z(A_{0}))$, we have $Y_{\sigma^{-1}(I)}((E_{2}\,X_{\sigma}))=a_{0}^{\sigma}b_{0}^{\sigma}\,P_{\sigma^{-1}(I)}(z^{\sigma}(A))$
and obtain the second equality. \end{proof}
\begin{defn}
\label{def:twist-G-E24}For $A\in M_{2,4}$ such that $[A]\in\mathbb{C}_{z}\cap\mathbb{C}_{z^{\sigma}}$,
we define the ratio of the factors in (\ref{eq:YA-by-affineY}) by
\begin{equation}
G(\sigma,e):=\frac{(\det B_{2}(\sigma))^{2}\,a_{0}^{\sigma}b_{0}^{\sigma}}{(\det B_{2}(e))^{2}\,a_{0}^{e}b_{0}^{e}}\;\left(=\frac{P_{I}(z(A))}{P_{\sigma^{-1}(I)}(z^{\sigma}(A))}\right),\label{eq:gaugeE24}
\end{equation}
and call it the \textit{twist factor} (or \textit{gauge factor}) for
the transition from $\mathbb{C}_{z}$ to $\mathbb{C}_{z^{\sigma}}$. 
\end{defn}
Explicitly, we calculate the twist factors $G(\sigma,e)$ in terms
of $z(A)=z$ for $[A]\in\mathbb{C}_{z}\cap\mathbb{C}_{z^{\sigma}}$
as follows: 
\begin{equation}
\begin{matrix}\sigma & : & 1 & (12) & (23) & (23)(12) & (12)(23) & (13)\\
G(\sigma,e) & : & 1 & z & 1-z & -z & z-1 & -1
\end{matrix}\label{eq:table-Gsigma}
\end{equation}

\begin{rem}
\label{rem:transf-period}The meaning of the twist factor becomes
clear in the definitions of period integrals (\ref{eq:periodE24-bar})
and (\ref{eq:periodE24-normalized}). Let us write the period integral
$\bar{\omega}_{C}(a)$ (\ref{eq:periodE24-bar}) by $\bar{\omega}(A)$.
Then, it is easy to see that the normalized period integral $\omega(z)$
in (\ref{eq:periodE24-normalized}) related to $\bar{\omega}(A)$
in general by 
\begin{equation}
\bar{\omega}(A)=\frac{1}{\det B_{2}(e)}\bar{\omega}\big((\begin{smallmatrix}1 & 0\\
0 & 1
\end{smallmatrix}\begin{smallmatrix}a_{1}^{e} & b_{0}^{e}\\
a_{0}^{e} & b_{1}^{e}
\end{smallmatrix})\big)=\frac{1}{\det B_{2}(e)}\frac{1}{\sqrt{a_{0}^{e}b_{0}^{e}}}\omega(z(A)).\label{eq:normalized-omega-A}
\end{equation}
We leave the derivations of the above relations for the reader. \end{rem}
\begin{lem}
\label{lem:omega-Sigma-omega}For $A\in M_{2,4}$ such that $[A]\in\mathbb{C}_{z}\cap\mathbb{C}_{z^{\sigma}}$,
the following relation holds 
\[
\omega(z^{\sigma}(A))=\sqrt{G(\sigma,e)}\,\omega(z(A))
\]
for the normalized period integral (\ref{eq:periodE24-normalized}). \end{lem}
\begin{proof}
By symmetry, the period integral $\bar{\omega}(A)$ in (\ref{eq:periodE24-bar})
satisfies the obvious relation $\bar{\omega}(A\sigma)=\bar{\omega}(A)$.
We can calculate $\bar{\omega}(A\sigma)=\bar{\omega}\big(B_{2}(\sigma)(\begin{smallmatrix}1 & 0\\
0 & 1
\end{smallmatrix}\begin{smallmatrix}a_{1}^{\sigma} & b_{0}^{\sigma}\\
a_{0}^{\sigma} & b_{1}^{\sigma}
\end{smallmatrix})\big)$ in the same way as (\ref{eq:normalized-omega-A}). Then the claimed
relation follows from $\bar{\omega}(A\sigma)=\bar{\omega}(A)$ and
the definition of the twist factor $G(\sigma,e)$ in (\ref{eq:gaugeE24}). \end{proof}
\begin{prop}
The Picard-Fuchs equation $\mathcal{D}_{z}\omega(z)=0$ transforms
to 
\begin{equation}
\mathcal{D}_{z^{\sigma}}\omega(z^{\sigma})=\bigg\{\theta_{z^{\sigma}}^{2}+z^{\sigma}(\theta_{z^{\sigma}}+\frac{1}{2})^{2}\bigg\}\omega(z^{\sigma})=0\label{eq:PF-ellip-zigma}
\end{equation}
under the twist $\omega(z^{\sigma})=\sqrt{G(\sigma,e)}\,\omega(z)$. \end{prop}
\begin{proof}
This is a consequence of Lemma \ref{lem:omega-Sigma-omega}. Also,
it is straightforward to verify the claim explicitly using $G(\sigma,e)$
and $z^{\sigma}=\varphi_{\sigma}(z)$ given in the tables (\ref{eq:table-Gsigma})
and (\ref{eq:table-zSigma}).
\end{proof}
It should be noted in the above proposition that the local solutions
about $z^{\sigma}=0$ have the same form for all three singularities.
In particular, the origins $z^{\sigma}=0$ are the so-called maximally
unipotent monodromy points (or LCSLs), which correspond to the cusps
in $\overline{\Gamma(2)\diagdown\mathbb{H}_{+}}$. This property comes
from the fact that the $\mathcal{D}$-modules of the Picard-Fuchs
equation around three singularities are all isomorphic. We will see
that similar properties hold for the double cover family of K3 surfaces
although the relevant $\mathcal{D}$-module becomes more complicated
(cf. Proposition \ref{prop:PF-o1}).

\subsection{The elliptic lambda function\label{sub:The-elliptic-lambda}}

We describe the elliptic lambda function (\ref{eq:elliptic-lambda})
by extending the projective relation 
\[
\Phi_{Y}([A])=\Phi\circ\mathcal{P}([A])\;\;\text{in }\mathbb{P}^{2}
\]
to an affine relation in $\mathbb{C}^{3}$. We will be brief since
the subject is more or less classical. However, for our definition
of K3 lambda functions, the corresponding affine relations will play
a central role.

\para{Transformation properties of theta functions} The theta functions
introduced in (\ref{para:theta-ell-def}) are modular forms of weight
two on $\Gamma(2)$. Let $S:\tau\to-1/\tau$, $T:\tau\to\tau+1$ be
the standard generators of $\Gamma=PSL(2,\mathbb{Z})$. The congruence
subgroup $\Gamma(2)$ is generated by $T^{2}$ and $ST^{2}S$. For
$g\in\Gamma$, which is given by composite of $S$ and $T$, we denote
its action on $\tau$ and $\theta_{k}(\tau)^{4}$ by $g\cdot\tau$
and $g\cdot\theta_{k}(\tau)^{4}=\theta_{k}(g\cdot\tau)^{4}$, respectively.
Then the transformation properties of the theta functions are determined
by 
\[
S\cdot\left(\begin{matrix}\theta_{2}(\tau)^{4}\\
\theta_{3}(\tau)^{4}\\
\theta_{4}(\tau)^{4}
\end{matrix}\right)=\tau^{2}\left(\begin{matrix}-\theta_{4}(\tau)^{4}\\
-\theta_{3}(\tau)^{4}\\
-\theta_{2}(\tau)^{4}
\end{matrix}\right),\;\;\;T\cdot\left(\begin{matrix}\theta_{2}(\tau)^{4}\\
\theta_{3}(\tau)^{4}\\
\theta_{4}(\tau)^{4}
\end{matrix}\right)=\left(\begin{matrix}-\theta_{2}(\tau)^{4}\\
\theta_{4}(\tau)^{4}\\
\theta_{3}(\tau)^{4}
\end{matrix}\right).
\]
 Denote by $\sigma_{g}$ the corresponding element of $g$ under a
group isomorphism $\Gamma/\Gamma(2)\simeq S_{3}$. When we fix the
isomorphism by $\sigma_{S}=(13)$ and $\sigma_{T}=(23)$, we can verify
that the above transformation properties become
\begin{equation}
\Theta\left(\begin{matrix}i\,\,j\\
k\,\,l
\end{matrix}\right)^{2}(g\cdot\tau)=(c\tau+d)^{2}\,\Theta\left(\begin{matrix}\sigma_{g}(i)\,\,\sigma_{g}(j)\\
\sigma_{g}(k)\,\,\sigma_{g}(l)
\end{matrix}\right)^{2}(\tau),\label{eq:Theta-sigma-tau}
\end{equation}
in the notation of Subsection \ref{para:theta-ell-def}, for $g=\left(\begin{smallmatrix}a & b\\
c & d
\end{smallmatrix}\right)$. We also use the inverse relation $\sigma\mapsto g_{\sigma}$ of
the isomorphism $\Gamma/\Gamma(2)\simeq S_{3}$. 

\para{The elliptic lambda function from the affine relation} The
period integral $\bar{\omega}(A)$ plays an important role in the
following arguments. 
\begin{prop}
\label{prop:Y-by-Ytilde-E24}For $A\in M_{2,4}$ such that $[A]\in\mathbb{C}_{z}\cap\mathbb{C}_{z^{\sigma}}$,
we have 
\begin{equation}
Y_{I}(A)\,\bar{\omega}(A)^{2}=P_{I}(z)\,\omega(z)^{2}=P_{\sigma^{-1}(I)}(z^{\sigma})\,\omega(z^{\sigma})^{2},\label{eq:z-to-zsigma-E24}
\end{equation}
 where $\bar{\omega}(A)$ is the period integral (\ref{eq:periodE24-bar})
and we set $z=z(A)$, $z^{\sigma}=z^{\sigma}(A)$.\end{prop}
\begin{proof}
The first equality follows from the first line of (\ref{eq:YA-by-affineY})
and the relation (\ref{eq:normalized-omega-A}). The second equality
follows from Lemma \ref{lem:omega-Sigma-omega} and (\ref{eq:gaugeE24})
.
\end{proof}
Note that this formal argument indicates that the product $Y_{I}(A)\,\bar{\omega}(A)^{2}$
depends only on the class $[A]\in\mathcal{M}_{4}$ and defines a holomorphic
function on $\mathcal{M}_{4}=\cup_{\sigma}\mathbb{C}_{z^{\sigma}}$.
It should be noted however that the product $Y_{I}(A)\,\bar{\omega}(A)^{2}$
is a multi-valued function which depends on the monodromy of the period
integral $\omega(z)^{2}$. More precisely, we can use the equality
(\ref{eq:z-to-zsigma-E24}) repeatedly from one chart to the other,
but after the analytic continuation along a closed path coming back
to $z\in\mathbb{C}_{z}$, we do not necessarily have the original
value $P_{I}(z)\omega(z)^{2}$ because of the monodromy of the period
integral $\omega(z)$. 

The monodromy of the hypergeometric series $\omega_{0}(z)$ (\ref{eq:w0-w1-E24})
has a particular form 
\[
\omega_{0}(z)\;\mapsto\;c\,\omega_{1}(z)+d\,\omega_{0}(z)=(c\,\tau+d)\,\omega_{0}(z)
\]
under the analytic continuation along a closed path $\mathcal{T}_{C}\left(\begin{smallmatrix}\omega_{1}(z)\\
\omega_{0}(z)
\end{smallmatrix}\right)=\left(\begin{smallmatrix}a & b\\
c & d
\end{smallmatrix}\right)\left(\begin{smallmatrix}\omega_{1}(z)\\
\omega_{0}(z)
\end{smallmatrix}\right)$ with $\left(\begin{smallmatrix}a & b\\
c & d
\end{smallmatrix}\right)\in\Gamma(2)$. We will not go into the detail, but only remark that this property
comes from the fact that $\omega_{0}(z)$ is a section of the Hodge
bundle over $\mathcal{M}_{4}$. 
\begin{prop}
\label{prop:solving-MS-E24}Let $\mathcal{P}_{z}=\mathcal{P}\vert_{\mathbb{C}_{z}}$
be the period map the $\mathcal{P}:\mathcal{M}_{4}\to\mathbb{H}_{+}$
restricted to $\mathbb{C}_{z}\subset\mathcal{M}_{4}$. Then the following
equality holds for $[A]\in\mathbb{C}_{z}$ : 
\begin{equation}
P_{I}(z(A))\,\omega_{0}(z(A))^{2}=-(-1)^{2}\Theta\left(\begin{matrix}i\,\,j\\
k\,\,l
\end{matrix}\right)^{2}(\mathcal{P}_{z}(A)),\label{eq:Affine-Key-relE24}
\end{equation}
where $\mathcal{P}_{z}(A)=\frac{\omega_{1}(z(A))}{\omega_{0}(z(A))}$
in terms of the hypergeometric series given in $($\ref{eq:w0-w1-E24}$)$
. This equality is analytically continued to the other chart $\mathbb{C}_{z^{\sigma}}\subset\mathcal{M}_{4}$,
giving the affine extension of the projective relation $\Phi_{Y}(A)=\Phi\circ\mathcal{P}(A)$
in (\ref{eq:E24diagram}). \end{prop}
\begin{proof}
The first claim is immediate writing (\ref{eq:Affine-Key-relE24})
explicitly. By definitions, we obtain three independent equations;
\[
(z-1)\omega_{0}^{2}=-\theta_{4}(\tau)^{4},\,\,\,(-1)\,\omega_{0}^{2}=-\theta_{3}(\tau)^{4},\,\,\,-z\,\omega_{0}^{2}=-\theta_{2}(\tau)^{4},
\]
where we $\tau:=\frac{\omega_{1}(z(A))}{\omega_{0}(z(A))}$. Solving
these equations for $z,\omega_{0}^{2}$, we obtain $z=\frac{\theta_{2}(\tau)^{4}}{\theta_{3}(\tau)^{4}}$
and $\omega_{0}^{2}(z)=\theta_{3}(\tau)^{4}$. The former is nothing
but the $\lambda(\tau)$ which is defined by the inverse relation
of $\tau:=\frac{\omega_{1}(z(A))}{\omega_{0}(z(A))}$. For the latter
equality to be consistent, we must have the identity 
\[
\omega_{0}(\lambda(\tau))^{2}=\theta_{3}(\tau)^{4},
\]
which is a well-known relation in the classical theory of hypergeometric
series. See \cite[Sect.5.4]{Zagier} for a modern formulation. 

The second claim follows the transformation property described in
(\ref{eq:z-to-zsigma-E24}). Assume $[A]\in\mathbb{C}_{z}$, then
$[A\sigma]\in\mbox{\ensuremath{\mathbb{C}}}_{z^{\sigma^{-1}}}$ and
we have $g_{\sigma}\cdot\tau=\mathcal{P}(A\sigma)$ since the diagram
(\ref{eq:E24diagram}) is equivariant under $\Gamma/\Gamma(2)\simeq S_{3}$
action. Now for $[A\sigma]\in\mbox{\ensuremath{\mathbb{C}}}_{z^{\sigma^{-1}}}$,
the affine relation (\ref{eq:Affine-Key-relE24}) is written by 
\[
P_{\sigma(I)}(z^{\sigma^{-1}}(A\sigma))\,\omega_{0}(z^{\sigma^{-1}}(A\sigma))^{2}=-\Theta\left(\begin{matrix}i\,\,j\\
k\,\,l
\end{matrix}\right)^{2}(g_{\sigma}\cdot\tau).
\]
Using $z^{\sigma^{-1}}(A\sigma)=z(A)$ and (\ref{eq:Theta-sigma-tau}),
we obtain 
\[
P_{\sigma(I)}(z(A))\,\omega_{0}(z(A))^{2}=-\Theta\left(\begin{matrix}\sigma(i)\,\,\sigma(j)\\
\sigma(k)\,\,\sigma(l)
\end{matrix}\right)^{2}(\tau),
\]
which is the relation imposed already on $\mathbb{C}_{z}$. Note that
the set of points of the form $[A\sigma]$ $([A]\in\mathbb{C}_{z})$
is a Zariski open subset of $\mathbb{C}_{z^{\sigma^{-1}}}$. Therefore
setting up the equations (\ref{eq:Affine-Key-relE24}) around $z=0$
($\mathrm{Im}\,\tau\gg1)$ automatically produces the corresponding
equations for all other affine chart $\mathbb{C}_{z^{\nu}}\,(\nu\in S_{3})$. 
\end{proof}
The affine relation (\ref{eq:Affine-Key-relE24}) is the one which
we will generalize to define the K3 lambda functions in the next section. 

~

~

\section{\textbf{\textup{The master equation for the $\lambda_{K3}$ functions}}
\label{sec:Master-Equation}}

\subsection{Double cover family of K3 surfaces }

Let us briefly recall the definition of a family of K3 surfaces branched
along six lines in general position in $\mathbb{P}^{2}$, which we
called \textit{double cover family of K3 surfaces} in \cite{HLTYpartI}.
We denote six lines in $\mathbb{P}^{2}$ by $\ell_{i}(i=1,...,6)$
with the following linear forms: 
\[
\ell_{i}(x,y,z):=a_{0i}z+a_{1i}x+a_{2i}y\;\;(i=1,...,6).
\]
When these lines are in general position, the double cover $\overline{X}\to\mathbb{P}^{2}$
branched along these six lines defines a singular K3 surface with
$A_{1}$ singularities at each 15 intersection points $P_{ij}:=\ell_{i}\cap\ell_{j}$.
Blowing-up these 15 $A_{1}$ singularities, we have a smooth K3 surface
$X$ of Picard number 16 generated by the hyperplane class $H$ from
$\mathbb{P}^{2}$ and the $(-2)$ curves of the exceptional divisors
$E_{ij}$ of the blow-up. The double cover family of K3 surfaces is
a (four dimensional) family of K3 surfaces over the configuration
space of six lines. The period integrals of this family and also their
monodromy properties were studied extensively in a paper \cite{YoshidaEtal}
by studying hypergeometric system $E(3,6)$. Also the configuration
space of six lines in $\mathbb{P}^{2}$ is a classical object in moduli
problems. It is known that the compactification via geometric invariant
theory \cite{DoOrt} is isomorphic to Baily-Borel-Satake compactification
\cite{vanDerGeer,Hunt}. We will denote this isomorphic compactified
moduli space by $\mathcal{M}_{6}$.

\subsection{Period integrals}

The double cover family of K3 surfaces is a natural generalization
of the Legendre family of elliptic curves. Corresponding to the period
integrals of the Legendre family, we have the period integrals of
holomorphic two forms 
\begin{equation}
\bar{\omega}_{C}(a)=\int_{C}\frac{d\mu}{\sqrt{\prod_{i=1}^{6}\ell_{i}(x,y,z)}},\label{eq:peridIntI}
\end{equation}
where $d\mu=xdy\wedge dz-ydx\wedge dz+zdx\wedge dy$, and $C$ are
integral (transcendental) cycles in $H_{2}(X,\mathbb{Z}$). The lattice
of transcendental cycles is known \cite{YoshidaEtal} to be 
\begin{equation}
T_{X}\simeq U(2)\oplus U(2)\oplus A_{1}\oplus A_{1},\label{eq:Tx}
\end{equation}
where $\mbox{\ensuremath{U}}$ represents the hyperbolic lattice of
rank 2, and $A_{1}=\langle-2\rangle$ is the root lattice of $sl(2,\mathbb{C})$.
The period integrals $\bar{\omega}_{C}(a)$ are parametrized by $3\times6$
matrix $A$ representing six lines in general positions as follows:
\[
A=\left(\begin{matrix}a_{01} & a_{02} & a_{03} & a_{04} & a_{05} & a_{06}\\
a_{11} & a_{12} & a_{13} & a_{14} & a_{15} & a_{16}\\
a_{21} & a_{22} & a_{23} & a_{24} & a_{25} & a_{26}
\end{matrix}\right).
\]
As in the preceding section, making the dependence on the cycles $C$
implicit, we often write the period integral simply by $\bar{\omega}(A)$.
Let $M_{3,6}$ be the affine space of all $3\times6$ matrices, and
set 
\[
M_{3,6}^{o}:=\left\{ A\in M_{3,6}\mid[i_{1}\,i_{2}\,i_{3}]\not=0\,(1\leq i_{1}<i_{2}<i_{3}\leq6)\right\} 
\]
 with $[i_{1},i_{2},i_{3}]$ representing $3\times3$ minors of $A$.
Then, under the genericity assumption, the configurations of six lines
are parametrized by 
\[
P(3,6):=GL(3,\mathbb{C})\diagdown M_{3,6}^{o}\diagup(\mathbb{C}^{*})^{6},
\]
where $(\mathbb{C}^{*})^{6}$ represents the diagonal $\mathbb{C}^{*}$-actions. 

Period integrals over the cycles define a multi-valued map, period
map, from $P(3,6)$ to the period domain 
\[
\mathcal{D}_{K3}=\left\{ [\omega]\in\mathbb{P}((U(2)^{\oplus2}\oplus A_{1}^{\oplus2})\otimes\mathbb{C}\mid\omega.\omega=0,\omega.\bar{\omega}>0\right\} ^{+},
\]
where $+$ represents one of the connected components. The period
map naturally extends to the compactified moduli space $\mathcal{M}_{6}$
of $P(3,6)$. In \cite{YoshidaEtal}, the monodromy group of the period
map has been determined to be the congruence subgroup $\mathcal{G}(2):=\left\{ g\in\mathcal{G}\mid g=E_{6}\text{ mod }2\right\} $
of 
\[
\mathcal{G}=\left\{ g\in PGL(6,\mathbb{Z})\mid\,^{t}gGg=G,H(g)>0\right\} ,
\]
where $G=\left(\begin{smallmatrix}0 & 2\\
2 & 0
\end{smallmatrix}\right)^{\oplus2}\oplus(-2)^{\oplus2}$ and $H(g)=(g_{11}+g_{12})(g_{23}+g_{34})-(g_{13}+g_{14})(g_{31}+g_{32})$.
The group $\mathcal{G}$ is a discrete subgroup of $\mathrm{Aut}(\mathcal{D}_{K3})$.
It is known \cite[Prop.\,2.8.2]{YoshidaEtal} that $\mathcal{G}/\mathcal{G}(2)\simeq S_{6}\times\mathbb{Z}_{2}$
for the quotient, where $S_{6}$ is the symmetric group of degree
six.

\subsection{Moduli space $\mathcal{M}_{6}$ and the period map}

The moduli space $\mathcal{M}_{6}$ is a well-studied object in many
contexts. We refer to \cite[Sect.2.3]{HLTYpartI} for a brief summary
on this space and references. Here we summarize some properties of
the moduli space and the period map of the family. 

\para{Baily-Borel-Satake compactification ${\mathcal M}_6$ and the period map}
The Baily-Borel-Satake compactification $\mathcal{M}_{6}$ is described
by an arithmetic quotient of the domain 
\[
\mathbb{H}_{2}=\left\{ W\in Mat(2,\mathbb{C})\mid(W^{\dagger}-W)/2i>0\right\} 
\]
where $W^{\dagger}:=\,^{t}\overline{W}$ . The Siegel half space $\frak{h}_{2}\subset\mathbb{H}_{2}$
is defined by $\,^{t}W=W$. Given a matrix $W\in\mathbb{H}_{2}$,
we have ten theta functions $\Theta_{i}(W)$ with even spin structures
(see Appendix \ref{sec:App-Theta-fn} for their explicit forms). With
these theta functions we define a map 
\[
\Phi:\mathbb{H}_{2}\to\mathbb{P}^{9},\;\;W\mapsto[\Theta_{1}(W)^{2},\Theta_{2}(W)^{2},\cdots,\Theta_{10}(W)^{2}]
\]
using the same letter $\Phi$ as in (\ref{eq:E24diagram}). These
squares of theta functions are modular forms of weight two on the
modular subgroup $\Gamma_{M}(1+i)(\simeq\mathcal{G}(2))$ of the discrete
subgroup $\Gamma_{T}(\simeq\mathcal{G})$ of $\mathrm{Aut}(\mathbb{H}_{2})(\simeq\mbox{\ensuremath{\mathrm{Aut}}}(\mathcal{D}_{K3}))$.
See \cite[Sect.\,3]{Matsu93} for more details. On the other hand,
using the semi-invariants $Y_{k}(A)$$(k=1,2,...,10)$ for the left
$GL(3,\mathbb{C})$ action on $3\times6$ matrices $A$, we have a
natural map $\Phi_{Y}:\mathcal{M}_{6}\to\mathbb{P}^{9}$ which gives
the following commutative diagram \cite[Thm.\,4.4.1]{Matsu93}: 
\begin{equation}
\xymatrix{\mathcal{M}_{6}\;\;\ar[dr]_{\Phi_{Y}}\ar[rr]^{\mathcal{P}} &  & \;\;\mathbb{H}_{2}\ar[dl]^{\Phi}\\
 & \mathbb{P}^{9}.
}
\label{eq:diagram-M6-H2-P9}
\end{equation}
As before, we code the semi-invarinats by the ordered partitions $I=\{\{i,j,k\},\{l,m,n\}\}$
of $\{1,2,...,6\}$ so that we have 
\begin{equation}
Y_{I}(A)=[ijk][lmn],\label{eq:def-Yijklmn}
\end{equation}
where the bracket $[i\,j\,k]$ represents the $3\times3$ minor of
$3\times6$ matrix of $A$ with the specified columns. We assume the
same sign changes of $Y_{I}$ under the permutations of $i,j,...,n$
as the r.h.s of (\ref{eq:def-Yijklmn}). Just as in the case of the
Legendre family, we shall take the relation 
\[
\Phi_{Y}([A])=\Phi\circ\mathcal{P}([A])\text{ in }\mathbb{P}^{9}
\]
as the guiding equation to define the K3 analogue of the lambda function.
One might expect that the same arguments as the elliptic lambda function
given in Section \ref{sub:The-elliptic-lambda} hold for the double
cover family of K3 surfaces. However, a crucial difference is that
the moduli space $\mathcal{M}_{6}$ is not smooth like $\mathcal{M}_{4}$.
To define the K3 lambda functions, we need to find suitable resolutions
of the singularity of $\mathcal{M}_{6}$ which we have done in \cite{HLTYpartI}. 

\para{Singularities of ${\mathcal M}_6$} It is known that $\mathcal{M}_{6}$
is singular along 15 lines of $A_{1}$ singularities. These lines
intersect at 15 points, each of which is given as a transversal intersection
of three lines. The configuration of these 15 lines is shown in Fig.
5 of \cite{HLTYpartI}. From the 15 lines, we can select a maximal
set of non-intersecting lines. Constructing the maximal set explicitly,
we see that every maximal set consists of 5 lines, and furthermore,
there are six possibilities for the maximal sets. 
\begin{prop}
The following properties hold:\end{prop}
\begin{enumerate}
\item $\Gamma_{T}/\Gamma_{M}(1+i)\simeq\mathcal{G}/\mathcal{G}(2)\simeq S_{6}\times\mathbb{Z}_{2}$.
\item The group $S_{6}$ acts on the six maximal set of non-intersecting
lines.
\item The group $S_{6}$ acts transitively on the 15 singular points.\end{enumerate}
\begin{proof}
We refer \cite[Prop.\,2.8.2]{YoshidaEtal} and also \cite[Prop.\,1.5.1]{Matsu93}
for (1). The properties (2),(3) are known in \cite[Prop.\,(1.1)]{vanDerGeer}. \end{proof}
\begin{prop}
The symmetric group $S_{6}$ in the preceding proposition is identified
with the natural $S_{6}$ action on $\mathcal{M}_{6}$ coming from
the action on $3\times6$ matrix $A$ from the right. Under this identification,
the diagram (\ref{eq:diagram-M6-H2-P9}) is $S_{6}$ equivariant. \end{prop}
\begin{proof}
The claims are shown in \cite{YoshidaEtal,Matsu93}. \end{proof}
\begin{lem}
\label{lem:local-X-sing}The singularities near the 0-dimensional
boundary points are locally isomorphic to the singularity near the
origin of
\[
\mathcal{X}:=\left\{ xyz-uv=0\right\} \subset\mathbb{C}^{5}.
\]
\end{lem}
\begin{proof}
This is proved in \cite[Props.4.4, 6.5]{HLTYpartI}
\end{proof}
In \cite{HLTYpartI}, we have described a resolution $\widetilde{\mathcal{X}}\to\mathcal{X}$
of the singularity, and also its (anti-)flip $\widetilde{\mathcal{X}}^{+}\to\mathcal{X}$.
The $E(3,6)$ system expressed by the local coordinates of these two
resolutions has a particularly nice property; there are LCSLs where
we can define the mirror maps, i.e., the lambda functions. We refer
to \cite{HLTYpartI} for more details of the resolutions.
\begin{prop}
\label{prop:S6-action-on-Mresol}The $S_{6}$ action on $\mathcal{M}_{6}$
extends to the resolutions $\widetilde{\mathcal{M}}_{6}$ and $\widetilde{\mathcal{M}}_{6}^{+}$. \end{prop}
\begin{proof}
The two resolution $\widetilde{\mathcal{M}}_{6}$ has been constructed
by blowing-up along the 15 lines of the singularity followed by blow-ups
at $2\times15$ points. Since the blowing-up at points are local,
they are compatible with the $S_{6}$ action. The (anti-)flip $\widetilde{\mathcal{M}}_{6}\to\widetilde{\mathcal{M}}_{6}^{+}$
is made by (anti-)flipping the local resolution $\widetilde{\mathcal{X}}\to\widetilde{\mathcal{X}}^{+}$
for all 15 isomorphic local geometry at one time. Hence, the resulting
flip $\widetilde{\mathcal{M}}_{6}^{+}$ retains the $S_{6}$ action
from $\widetilde{\mathcal{M}}_{6}$. 
\end{proof}
Let us recall the following covering property \cite{HLTYpartI}: 
\begin{equation}
\mathcal{M}_{6}=\bigcup_{\sigma\in S_{6}}\phi_{\sigma}(\mathcal{M}_{3,3}\setminus D_{0}),\label{eq:M6-sigma-Union}
\end{equation}
where $\mathcal{M}_{3,3}$ is a toric hypersurface in $\mathbb{P}^{5}$
which is birational to $\mathcal{M}_{6}$, and $D_{0}$ is a divisor
in $\mathcal{M}_{3,3}$. The toric hypersurface $\mathcal{M}_{3,3}$
is singular along 9 lines of $A_{1}$ singularity, and these lines
intersect at 6 points (cf. Lemma \ref{lem:local-X-sing}). 

We will define our lambda functions, first locally, by the mirror
maps given in the form of $q$-expansions near the LCSLs in the local
resolutions $\widetilde{\mathcal{X}}$ (or $\widetilde{\mathcal{X}}^{+}$)
of $\mathcal{X}$. Then we will show that these local definitions
actually extend to a global definition. To ensure that, we use Proposition
\ref{prop:S6-action-on-Mresol} and the transformation property of
some local expressions under the $S_{6}$ action. This is exactly
parallel to the one we presented in the preceding section for the
elliptic lambda function.

\subsection{Defining $\lambda_{K3}$ functions }

Recall that the equation (\ref{eq:Affine-Key-relE24}) comes from
the commutative diagram (\ref{eq:E24diagram}). We generalize this
for the corresponding diagram (\ref{eq:diagram-M6-H2-P9}).

\para{LCSLs in $\widetilde{\mathcal{X}}$ and $\widetilde{\mathcal{X}}^+$}
For simplicity, let us write 
\[
\phi_{\sigma}(\mathcal{M}_{3,3}^{D_{0}}):=\phi_{\sigma}(\mathcal{M}_{3,3}\setminus D_{0})
\]
in the decomposition (\ref{eq:M6-sigma-Union}) of $\mathcal{M}_{6}$.
Since the component $\phi_{\sigma}(\mathcal{M}_{3,3}^{D_{0}})\subset\mathcal{M}_{6}$
is isomorphic to a Zariski open subset of a toric variety $\mathcal{M}_{3,3}$,
a general point $p\in\phi_{\sigma}(\mathcal{M}_{3,3}^{D_{0}})$ is
represented by a $3\times6$ matrix $A$ having the properties 
\begin{equation}
A\sigma=B_{3}(\sigma)\left(\;\;E_{3}\;\;\;\;\begin{matrix}a_{2}^{\sigma} & b_{1}^{\sigma} & c_{0}^{\sigma}\\
a_{0}^{\sigma} & b_{2}^{\sigma} & c_{1}^{\sigma}\\
a_{1}^{\sigma} & b_{0}^{\sigma} & c_{2}^{\sigma}
\end{matrix}\right),\;\;\prod_{i=0}^{2}a_{i}^{\sigma}b_{i}^{\sigma}c_{i}^{\sigma}\not=0,\label{eq:Asigma-E3}
\end{equation}
for a unique $B_{3}(\sigma)\in GL(3,\mathbb{C})$. The open subset
$\phi_{e}(\mathcal{M}_{3,3}^{D_{0}})$ contains six copies of the
local geometry $\mathcal{X}=\left\{ xyz=uv\right\} $. We will identify
one of them with $\mathcal{X}$, and denote it by $\mathcal{X}_{e}$.
We denote its resolutions by $\widetilde{\mathcal{X}}_{e}\to\mathcal{X}_{e}$
and $\widetilde{\mathcal{X}}_{e}^{+}\to\mathcal{X}_{e}$. 

\begin{figure}[h] 
\resizebox{8.5cm}{!}{
\includegraphics[trim=0cm 1.5cm 0cm 2.5cm,
clip=true]{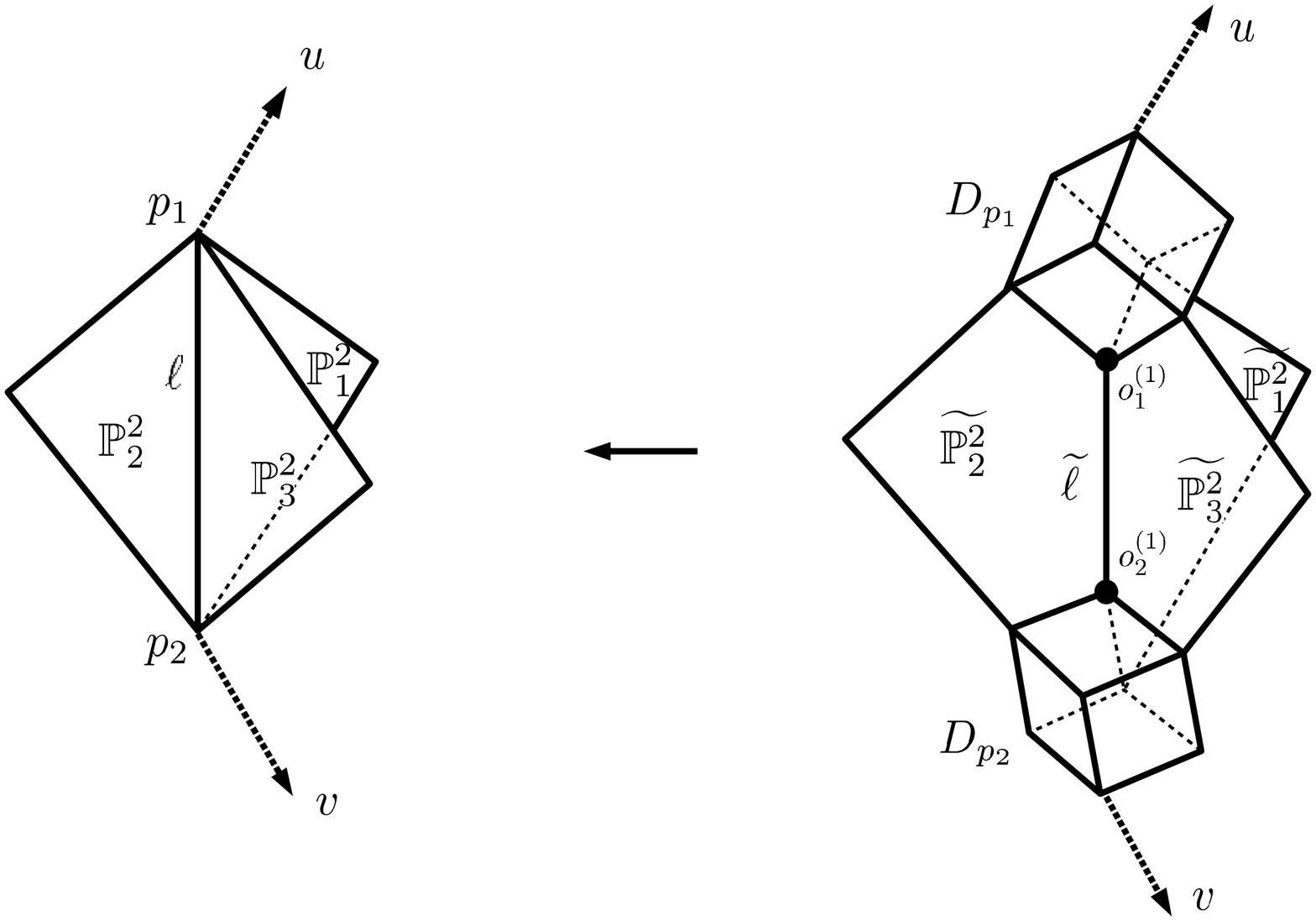}}
\caption{{\bf Fig.1} The blow-up of $\mathcal{X}'$ at $p_1$, $p_2$. The two points $o_1^{(1)},o_2^{(2)}$ are LCSLs. }
\end{figure}In Fig.1, it is shown that the resolution $\widetilde{\mathcal{X}}$
contains two LCSLs, $o_{i}^{(1)}(i=1,2)$. The left figure of Fig.1
is the blow-up $\mathcal{X}'\to\mathcal{X}$ along three coordinate
axes that introduces corresponding exceptional divisors $E_{x},E_{y},E_{z}$.
The right figure represents the resolution $\widetilde{\mathcal{X}}=\widetilde{\mathcal{X}}_{e}$
by the blow-up at two points $p_{1},p_{2}$ which introduces the exceptional
divisors $D_{p_{k}}(k=1,2)$. The two LCSLs are given by the intersections
$o_{k}^{(1)}=\widetilde{E}_{x}\cap\widetilde{E}_{y}\cap\widetilde{E}_{z}\cap D_{p_{k}}$
of $D_{p_{k}}$ and the proper transforms of the three exceptional
divisors. We introduce local coordinate $z_{k}(o_{1})(k=1,...,4)$
near the point $o_{1}^{(1)}$ so that 
\[
z_{1}(o_{1})=0,\;\;z_{2}(o_{1})=0,\;\;z_{3}(o_{1})=0\text{ and }\;z_{4}(o_{1})=0
\]
are the local equations for the divisors $\widetilde{E}_{x},\widetilde{E}_{y},\widetilde{E}_{z}$
and $D_{p_{1}}$, respectively. Near the other point $o_{2}^{(1)},$
we introduce local coordinate $z_{k}(o_{2})(k=1,...,4)$ in a similar
way except that $z_{4}(o_{2})=0$ represents the divisor $D_{p_{2}}$. 

The transforms $\phi_{\sigma}(\mathcal{X})\subset\phi_{\sigma}(\mathcal{M}_{3,3}^{D_{0}})$
of the local geometry $\phi(\mathcal{X})\subset\phi(\mathcal{M}_{3,3}^{D_{0}})$
by the automorphisms $\varphi_{\sigma}\in\mathrm{Aut}(\mathcal{M}_{6})$,
where $\phi_{\sigma}:=\varphi_{\sigma}\circ\phi$ \cite[Def. 6.8]{HLTYpartI},
are all isomorphic. We set $\mathcal{X}_{\sigma}:=\phi_{\sigma}(\mathcal{X})$
and denote by $\widetilde{\mathcal{X}}_{\sigma}\to\mathcal{X}_{\sigma}$
the resolution which is isomorphic to the resolution $\widetilde{\mathcal{X}}\to\mathcal{X}$.
Similarly for the other resolution $\widetilde{\mathcal{X}}_{\sigma}^{+}\to\mathcal{X}_{\sigma}$.
We denote by $z_{k}^{\sigma}(o_{i})$ and $z_{k}^{\sigma}(o_{i}^{+})$,
respectively, the corresponding local coordinates of the resolutions
$\widetilde{\mathcal{X}}_{\sigma}\to\mathcal{X}_{\sigma}$ and $\widetilde{\mathcal{X}}_{\sigma}^{+}\to\mathcal{X}_{\sigma}$.
When $\sigma=e$, we often omit the superscript, e.g., $z_{k}^{e}=z_{k}$. 
\begin{prop}
\label{prop:zk-O1}The coordinate functions $z_{k}^{\sigma}(o_{1})$
evaluate the general points $p\in\phi_{\sigma}(\mathcal{M}_{3,3}^{D_{0}})$
by 
\begin{equation}
z_{1}^{\sigma}(o_{1})=-\frac{a_{1}^{\sigma}c_{1}^{\sigma}}{a_{0}^{\sigma}c_{2}^{\sigma}},\;\;z_{2}^{\sigma}(o_{1})=-\frac{a_{1}^{\sigma}b_{1}^{\sigma}}{a_{2}^{\sigma}b_{0}^{\sigma}},\;\;z_{3}^{\sigma}(o_{1})=-\frac{b_{1}^{\sigma}c_{1}^{\sigma}}{b_{2}^{\sigma}c_{0}^{\sigma}},\;\;z_{4}^{\sigma}(o_{1})=\frac{a_{2}^{\sigma}b_{2}^{\sigma}c_{2}^{\sigma}}{a_{1}^{\sigma}b_{1}^{\sigma}c_{1}^{\sigma}},\label{eq:zk-o2}
\end{equation}
where $a_{i}^{\sigma},b_{i}^{\sigma},c_{i}^{\sigma}$ are determined
by (\ref{eq:Asigma-E3}) from $p$ by choosing a matrix $A$ such
that $p=[A]$. These are independent of the representative $A$ of
$p$. \end{prop}
\begin{proof}
The coordinate functions $z_{k}^{\sigma}(o_{1})$ are determined in
Lemma 3.3 of \cite{HLTYpartI} as the generators of the coordinate
ring $\mathbb{C}[(\sigma_{1}^{(1)})^{\vee}\cap L]$ of the affine
coordinate around $o_{1}^{(1)}$ of the resolution $\widetilde{\mathcal{X}}_{\sigma}$.
For a matrix $A$ such that $p=[A]$, assume the matrix $A\sigma$
has the form (\ref{eq:Asigma-E3}), then the claimed form (\ref{eq:zk-o2})
follows by the definitions given in \cite[Sect.\,3.2.b]{HLTYpartI}.
Let us write $A\sigma=B_{3}(\sigma)(E_{3}\,X_{\sigma})$. If we change
$A$ to $gAt$ by $(g,t)\in GL(3,\mathbb{C})\times(\mathbb{C}^{*})^{6}$,
then we have $(gAt)\sigma=gB_{3}(\sigma)(E_{3}\,X_{\sigma})t_{\sigma}$
with $t_{\sigma}:=\sigma^{-1}t\sigma$. It is easy to find a diagonal
matrix $h\in GL(3,\mathbb{C})$ satisfying $h(E_{3}\,X_{\sigma})t_{\sigma}=(E_{3}\,X_{\sigma}')$,
and we have $(gAt)\sigma=gB_{3}(\sigma)h^{-1}(E_{3}\,X_{\sigma}')$.
Since $h(E_{3}\,X_{\sigma})t_{\sigma}=(E_{3}\,X_{\sigma}')$ is the
torus action on $X_{\sigma}$ described in \cite[Sect.\,2.4.a]{HLTYpartI}
(see also Sect.$\,$\ref{sub:GKZ-Def} below), the values of the coordinate
functions $z_{k}^{\sigma}(o_{1})$ do not change for $X_{\sigma}$
and $X_{\sigma}'$. \end{proof}
\begin{prop}
The coordinate functions $z_{k}^{\sigma}(o_{2})$ are related to $z_{k}^{\sigma}(o_{1})$
by 
\begin{equation}
z_{k}^{\sigma}(o_{2})=z_{k}^{\sigma}(o_{1})z_{4}^{\sigma}(o_{1})\,\,(k=1,2,3)\text{ and }z_{4}^{\sigma}(o_{2})=\frac{1}{z_{4}^{\sigma}(o_{1})}.\label{eq:z-o2-by-o1}
\end{equation}
\end{prop}
\begin{proof}
This follows directly from Lemma 3.3 of \cite{HLTYpartI}. Using the
definitions there, the generators of the cone $(\sigma_{1}^{(1)})^{\vee}\cap L$\textcolor{red}{{}
}determine the coordinate functions $z_{k}(o_{1})$. We read the coordinate
functions $z_{k}^{\sigma}(o_{2})$ from the generators of the cone
$(\sigma_{2}^{(1)})^{\vee}\cap L$. \end{proof}
\begin{rem}
Clearly, the decomposition (\ref{eq:M6-sigma-Union}) is a generalization
of the corresponding decomposition (\ref{eq:M4-sigma-Union}) of $\mathcal{M}_{4}$.
By similar arguments done for the coordinate functions on $\mathbb{C}_{\sigma}\subset\mathcal{M}_{4}$,
the coordinate functions $z_{k}^{\sigma}(o_{i})$ and $z_{k}(o_{i})$
are related by $\varphi_{\sigma}:=\phi_{\sigma}\circ\phi_{e}^{-1}\in\mathrm{Aut}(\mathcal{M}_{6})$
for $[A]\in\phi_{\sigma}(\mathcal{M}_{3,3}^{D_{0}})\cap\phi_{e}(\mathcal{M}_{3,3}^{D_{0}})$
(see Subsection~\ref{para:affine-coordinate-ell}). This generalizes
the classical representation (\ref{eq:table-zSigma}) of $S_{3}$
on $\mathrm{Aut}(\mathcal{M}_{4})\simeq\mathrm{Aut}(\mathbb{P}^{1})$.
Unfortunately, the relations $z_{k}^{\sigma}(o_{i})=\varphi_{\sigma}(z_{1}(o_{i}),...,z_{4}(o_{i}))$
are not simple enough to list them in a table. In Appendix \ref{sec:App-representation-S6},
we show them explicitly for some $\sigma\in S_{6}$. 
\end{rem}
The flipped resolution $\widetilde{\mathcal{X}}_{e}^{+}\to\mathcal{X}_{e}$
contains three LCSLs, $o_{i}^{+}(i=1,2,3)$ which arises as the transversal
intersections of four divisors. We will denote the corresponding coordinates
by $z_{k}(o_{1}^{+}),z_{k}(o_{2}^{+})$ and $z_{k}(o_{3}^{+})$ for
$o_{i}^{+}(i=1,2,3)$. See Appendix \ref{sec:App.PF-equation-Bo1}
for the explicit descriptions of these local coordinates. 

\para{Semi-invariants in the affine coordinates $z_k$} In what follow,
we will focus on the boundary points given by $z_{1}^{\sigma}(o_{1})=\cdots=z_{4}^{\sigma}(o_{1})=0$.
For simplicity, we write $z_{k}^{\sigma}(o_{1})$ by $z_{k}^{\sigma}$
unless otherwise stated. Also we write $z_{k}^{e}$ by $z_{k}$. Then,
for a general matrix $A$, the expression $z_{k}(A)=z_{k}^{e}(A)$
represents the ratios (\ref{eq:zk-o2}) defined by making the matrix
$A=Ae$ into the form (\ref{eq:Asigma-E3}). As in the table (\ref{eq:table-zSigma})
in the preceding section, we have 
\[
z_{k}^{\sigma}(A)=z_{k}(A\sigma)=\varphi_{\sigma}(z_{1}(A),...,z_{4}(A)).
\]

\begin{defn}
Take a special form $A_{0}=\left(\;\;E_{3}\;\;\;\begin{smallmatrix}a_{2} & b_{1} & c_{0}\\
a_{0} & b_{2} & c_{1}\\
a_{1} & b_{0} & c_{2}
\end{smallmatrix}\right)$ with $a_{i},b_{i},c_{i}\in\mathbb{C}^{*}$. Using this, we define
affine semi-invariants $P_{I}$ by 
\[
P_{I}:=\frac{1}{a_{0}b_{0}c_{0}}Y_{I}(A_{0}),
\]
where $Y_{I}$ is the semi-invariants in (\ref{eq:diagram-M6-H2-P9}). 
\end{defn}
The above definition is parallel to the case of $\mathcal{M}_{4}$.
It is straightforward to find that these affine semi-invariants are
polynomial functions of $z_{k}$ (defined for $\sigma=e$). See Appendix
\ref{sec:App-P-Q} for their explicit expressions. 
\begin{prop}
For a general matrix $A$ such that $[A]\in\phi_{e}(\mathcal{M}_{3,3}^{D_{0}})\cap\phi_{\sigma}(\mathcal{M}_{3,3}^{D_{0}})$,
the following equalities hold:
\[
\begin{aligned}Y_{I}(A) & =(\det B_{3}(e))^{2}a_{0}^{e}b_{0}^{e}c_{0}^{e}\cdot P_{I}(z(A))\\
 & =(\det B_{3}(\sigma))^{2}a_{0}^{\sigma}b_{0}^{\sigma}c_{0}^{\sigma}\cdot P_{\sigma^{-1}(I)}(z^{\sigma}(A)).
\end{aligned}
\]
\end{prop}
\begin{proof}
Since the derivations are parallel to Proposition \ref{prop:Y-P-rel-E24},
we omit them here.\end{proof}
\begin{defn}
\label{def:twist-G-M6}For a general matrix $A$ such that $[A]\in\phi_{e}(\mathcal{M}_{3,3}^{D_{0}})\cap\phi_{\sigma}(\mathcal{M}_{3,3}^{D_{0}})$,
we define 
\[
G(\sigma,e):=\frac{(\det B_{3}(\sigma))^{2}\,a_{0}^{\sigma}b_{0}^{\sigma}c_{0}^{\sigma}}{(\det B_{3}(e))^{2}\,a_{0}^{e}b_{0}^{e}c_{0}^{e}}\;\left(=\frac{P_{I}(z(A))}{P_{\sigma^{-1}(I)}(z^{\sigma}(A))}\right),
\]
and call this a\textit{ twist factor} (cf. Definition \ref{def:twist-G-E24}).
We also set $G(\sigma,\tau):=G(\sigma,e)/G(\tau,e)$. 
\end{defn}
The following definition coincides with the normalized period integral
\cite[(3.4)]{HLTYpartI} which corresponds to (\ref{eq:periodE24-normalized}). 
\begin{defn}
For a general matrix $A$ such that $[A]\in\phi_{\sigma}(\mathcal{M}_{3,3}^{D_{0}})$,
we define the normalized period integral 
\begin{equation}
\omega(z^{\sigma}(A))=\det\,B_{3}(\sigma)\sqrt{a_{0}^{\sigma}b_{0}^{\sigma}c_{0}^{\sigma}}\,\bar{\omega}(A).\label{eq:period-E36-normalized}
\end{equation}

\end{defn}
We leave the reader to show that the right hand side of (\ref{eq:period-E36-normalized})
is a function of $z^{\sigma}(A)$ (cf. \cite[Sect.\,4]{HLTYpartI}). 

\para{The master equation for the $\lambda_{K3}$ functions} We introduce
the master equation by which we define the $\lambda_{K3}$ functions. 
\begin{prop}
For a general matrix $A$ such that $[A]\in\phi_{e}(\mathcal{M}_{3,3}^{D_{0}})\cap\phi_{\sigma}(\mathcal{M}_{3,3}^{D_{0}})$,
we have 
\begin{equation}
Y_{I}(A)\,\bar{\omega}(A)^{2}=P_{I}(z)\,\omega(z)^{2}=P_{\sigma^{-1}(I)}(z^{\sigma})\,\omega(z^{\sigma})^{2},\label{eq:Y-PI-E36}
\end{equation}
 where $\bar{\omega}(A)$ is the period integral (\ref{eq:peridIntI})
and we set $z_{k}=z_{k}(A)$, $z_{k}^{\sigma}=z_{k}^{\sigma}(A)$.\end{prop}
\begin{proof}
Derivations are parallel to Proposition \ref{prop:Y-by-Ytilde-E24}.
\end{proof}
Now we extend the projective relation $\Phi_{Y}([A])=\Phi\circ\mathcal{P}([A])$
in (\ref{eq:diagram-M6-H2-P9}) to 
\begin{defn}[\textbf{Master Equation}]
 For a general matrix $A$ such that $[A]\in\phi_{e}(\mathcal{M}_{3,3}^{D_{0}})$,
we define\textcolor{red}{{} }
\begin{equation}
P_{I}(z(A))\,\omega(z(A))^{2}=-(-1)^{3}\Theta^{\sigma}\left(\begin{matrix}i\,\,j\,\,k\\
l\,\,m\,\,n
\end{matrix}\right)^{2}(\mathcal{P}([A])),\label{eq:Master-Eq}
\end{equation}
where $\Theta^{\sigma}\left(\begin{smallmatrix}i\,j\,k\\
l\,m\,n
\end{smallmatrix}\right):=\Theta\left(\begin{smallmatrix}\sigma(i)\,\sigma(j)\,\sigma(k)\\
\sigma(l)\,\sigma(m)\,\sigma(n)
\end{smallmatrix}\right)$ represents a possible permutation of the labels from the $\Theta$
defined in \cite{Matsu93} (see also Appendix \ref{sec:App-Theta-fn}). 
\end{defn}
\vskip0.3cm

The master equation (\ref{eq:Master-Eq}) generalizes the equation
(\ref{eq:Affine-Key-relE24}) which characterizes the elliptic lambda
function $\lambda(\tau)$ together with the classical relation $\omega_{0}(\lambda(\tau))^{2}=\theta_{3}(\tau)^{4}$
for the hypergeometric series (\ref{eq:w0-elliptic}). In the following
sections, we will find that the mirror map and the unique (up to constant)
hypergeometric series $\omega(z)=\omega_{0}(z)$ near the LCSL satisfy
the above master equation (Subsection \ref{sub:masterEq-1} and Subsection
\ref{sub:masterEq-2}). 
\begin{rem}
\label{rem:Y-Y'-Aphase}When we use the local coordinates $z_{k}(o_{2})$
for the other LCSL in the resolution $\widetilde{\mathcal{X}}\subset\widetilde{\mathcal{M}}_{6}$,
we will have the master equation in the same form as above. However,
the polynomials of $P_{I}'=\frac{1}{a_{0}b_{0}c_{0}}Y_{I}(A_{0})$
by the coordinate $z_{k}(o_{2})$ differ from those given above. Namely,
we have 
\begin{equation}
\frac{1}{a_{0}b_{0}c_{0}}Y_{I}(A_{0})=P_{I}(z(o_{1}))=P_{I}'(z(o_{2}))\label{eq:Pz=00003DP'zo2}
\end{equation}
with different polynomials $P_{I}(t)$ and $P_{I}'(t)$ for the coordinates
$z_{k}=z_{k}(o_{1})$ and $z_{k}(o_{2})$, respectively. It is straightforward
to see that the following simple relation holds: 
\begin{equation}
P_{I}'(z(o_{2}))=P_{\alpha(I)}(z(o_{2}))\text{ with }\alpha=(16)(24)(35).\label{eq:P'=00003DPalpha}
\end{equation}

~

~
\end{rem}

\section{\textbf{\textup{Generalized Frobenius method for local solutions
\label{sec:Gen-Frobenius-method}}}}

As studied in \cite{HKTY,HLY}, the GKZ hypergeometric systems in
mirror symmetry are resonant and the mirror correspondence is encoded
in the special form of local solutions expressed using Frobenius method,
which generalizes the classical method for ordinary hypergeometric
differential equations to GKZ hypergeometric systems of multi-variables.
Since several new features can be observed in this generalization,
e.g. Remark \ref{rem:Remark-GKZ-Frob} below, we will call it the
\textit{generalized Frobenius method} when we emphasize them.

\subsection{GKZ hypergeometric system from $E(3,6)$\label{sub:GKZ-Def}}

At least locally, following \cite{HKTY,HLY}, we can describe the
period map $\mathcal{P}\to\mathbb{H}_{2}$ by using local solutions
of Picard-Fuchs equation near the LCSL $z_{1}(o_{1})=\cdots=z_{4}(o_{1})=0.$
As in the preceding section, restricting our attentions to the neighborhood
of $o_{1}$, we simply write by $z_{k}$ for $z_{k}(o_{1})$. 

The Picard-Fuchs differential operators have been determined from
the GKZ system associated to the $E(3,6)$ system \cite{HLTYpartI}.
The GKZ system arises form the $E(3,6)$ system by taking the following
special form of a matrix $A\in M_{3,6}$: 
\begin{equation}
A_{0}:=\left(\begin{matrix}1 & 0 & 0 & a_{2} & b_{1} & c_{0}\\
0 & 1 & 0 & a_{0} & b_{2} & c_{1}\\
0 & 0 & 1 & a_{1} & b_{0} & c_{2}
\end{matrix}\right)=:(E_{3}\,\bm{a}\,\bm{b}\,\bm{c}).\label{eq:E3-gaugeA}
\end{equation}
This form reduces the left $GL(3,\mathbb{C})$ action on $A$ to the
residual subgroup action of the diagonal tori $(\mathbb{C}^{*})^{3}\subset GL(3,\mathbb{C})$.
Taking into account the $(\mathbb{C}^{*})^{6}$ action from the right,
we define 
\[
T:=\left\{ (g,t)\in GL(3,\mathbb{C})\times(\mathbb{C}^{*})^{6}\mid g\left(E_{3}\begin{smallmatrix}* & * & *\\
* & * & *\\
* & * & *
\end{smallmatrix}\right)t=\left(E_{3}\begin{smallmatrix}* & * & *\\
* & * & *\\
* & * & *
\end{smallmatrix}\right)\right\} /\sim,
\]
where $(g,t)\sim(g\lambda,\lambda^{-1}t)$ with $(\lambda\in\mathbb{C}^{*}).$
For the matrix $A_{0}$, we have the following form of the period
integral: 

\begin{equation}
\begin{aligned}\bar{\omega}(A_{0}) & =\int\frac{dx\wedge dy}{\sqrt{xy(a_{2}+a_{0}x+a_{1}y)(b_{1}+b_{2}x+b_{0}y)(c_{0}+c_{1}x+c_{2}y)}}\\
 & =\int\frac{1}{\sqrt{\big(a_{0}+a_{1}\frac{y}{x}+\frac{a_{2}}{x}\big)\big(b_{0}+\frac{b_{1}}{y}+b_{2}\frac{x}{y}\big)\big(c_{0}+c_{1}x+c_{2}y\big)}}\frac{dx\wedge dy}{xy}.
\end{aligned}
\label{eq:1-over-f1f2f3}
\end{equation}
Recognizing a striking similarities of (\ref{eq:1-over-f1f2f3}) with
the equations we encountered in \cite{HKTY}, we observed that the
period integral satisfies GKZ $\mathcal{A}$-hypergeometric system
with a suitable choice of the finite set $\mathcal{A}$ (see \cite[Prop.3.1]{HLTYpartI}). 

For a general matrix $A\in\phi_{e}(\mathcal{M}_{3,3}^{D_{0}})$, it
holds that $\det\,B_{3}(e)\not=0$ and $\bar{\omega}(A_{0})=\det\,B_{3}(e)\,\bar{\omega}(A)$.
The normalized period integral (\ref{eq:period-E36-normalized}) is
given by 
\[
\omega(z)=\sqrt{a_{0}b_{0}c_{0}}\,\bar{\omega}(A_{0})=\det\,B_{3}(e)\sqrt{a_{0}b_{0}c_{0}}\,\bar{\omega}(A),
\]
where we should identify $a_{0},b_{0},c_{0}$, respectively, with
$a_{0}^{e},b_{0}^{e},c_{0}^{e}$ (cf. (\ref{eq:Asigma-E3})), and
we have chosen the affine coordinate $z^{e}(A)=z$ centered at the
LCSL $o_{1}$. The following proposition is described in \cite[Prop. 3.6, Appendix C]{HLTYpartI}:
\begin{prop}
\label{prop:PF-o1}The normalized period integral satisfies the Picard-Fuchs
system which consists of differential equations $\mathcal{D}_{i}\omega(z)=0(i=1,...,9)$
with 
\begin{equation}
\begin{matrix}\mathcal{D}_{1}=(\theta_{1}+\theta_{2}-\theta_{4})(\theta_{1}+\theta_{3}-\theta_{4})+z_{1}(\theta_{1}+\frac{1}{2})(\theta_{1}-\theta_{4}),\\
\mathcal{D}_{2}=(\theta_{1}+\theta_{2}-\theta_{4})(\theta_{2}+\theta_{3}-\theta_{4})+z_{2}(\theta_{2}+\frac{1}{2})(\theta_{2}-\theta_{4}),\\
\mathcal{D}_{3}=(\theta_{1}+\theta_{3}-\theta_{4})(\theta_{2}+\theta_{3}-\theta_{4})+z_{3}(\theta_{3}+\frac{1}{2})(\theta_{3}-\theta_{4}),\\
\mathcal{D}_{4}=(\theta_{2}-\theta_{4})(\theta_{3}-\theta_{4})-z_{1}z_{4}(\theta_{1}+\frac{1}{2})(\theta_{2}+\theta_{3}-\theta_{4}),\;\;\;\;\\
\mathcal{D}_{5}=(\theta_{1}-\theta_{4})(\theta_{3}-\theta_{4})-z_{2}z_{4}(\theta_{2}+\frac{1}{2})(\theta_{1}+\theta_{3}-\theta_{4}),\;\;\;\;\\
\mathcal{D}_{6}=(\theta_{1}-\theta_{4})(\theta_{2}-\theta_{4})-z_{3}z_{4}(\theta_{3}+\frac{1}{2})(\theta_{1}+\theta_{2}-\theta_{4}),\;\;\;\;\\
\mathcal{D}_{7}=(\theta_{1}+\theta_{2}-\theta_{4})(\theta_{3}-\theta_{4})+z_{1}z_{2}z_{4}(\theta_{1}+\frac{1}{2})(\theta_{2}+\frac{1}{2}),\;\;\\
\mathcal{D}_{8}=(\theta_{1}+\theta_{3}-\theta_{4})(\theta_{2}-\theta_{4})+z_{1}z_{3}z_{4}(\theta_{1}+\frac{1}{2})(\theta_{3}+\frac{1}{2}),\;\;\\
\mathcal{D}_{9}=(\theta_{2}+\theta_{3}-\theta_{4})(\theta_{1}-\theta_{4})+z_{2}z_{3}z_{4}(\theta_{2}+\frac{1}{2})(\theta_{3}+\frac{1}{2}),\;\;
\end{matrix}\label{eq:PF-system-o1}
\end{equation}
where $\theta_{i}:=z_{i}\frac{\partial}{\partial z_{i}}$. Around
the origin $z_{1}=\cdots=z_{4}=0,$ this system admits only one (up
to constant) regular solution given by 
\[
\omega_{0}(z)=\sum_{n_{1},n_{2},n_{3},n_{4}\geq0}c(n_{1},n_{2},n_{3},n_{4})z_{1}^{n_{1}}z_{2}^{n_{2}}z_{3}^{n_{3}}z_{4}^{n_{4}}
\]
with the coefficients $c(n_{1},n_{2},n_{3},n_{4})=:c(n)$ given by
\begin{equation}
c(n)=\frac{1}{\Gamma(\frac{1}{2})^{3}}\frac{\Gamma(n_{1}+\frac{1}{2})\Gamma(n_{2}+\frac{1}{2})\Gamma(n_{3}+\frac{1}{2})}{\Pi_{_{i=1}}^{3}\Gamma(n_{4}-n_{i}+1)\cdot\Pi_{1\leq j<k\leq3}\Gamma(n_{j}+n_{k}-n_{4}+1)}.\label{eq:def-cn}
\end{equation}
\end{prop}
\begin{rem}
The Picard-Fuchs system around the other point $o_{2}$ of the resolution
$\widetilde{\mathcal{X}}_{e}\to\mathcal{X}_{e}$ simply follows from
(\ref{eq:PF-system-o1}) by using the monomial relation (\ref{eq:z-o2-by-o1}).
The Picard-Fuchs systems around the points $o_{1}^{+}$ of the $\widetilde{\mathcal{X}}_{e}^{+}\to\mathcal{X}_{e}$
are described in Appendix \ref{sec:App.PF-equation-Bo1}. The systems
for the other boundary points $o_{2}^{+},o_{3}^{+}$ follow from the
above system by using the monomial relations described there. 
\end{rem}

\subsection{Period integrals by the generalized Frobenius method\label{sub:PF-and-Frobenius}}

The Picard-Fuchs system is a complete set of differential equations
which determines all local solutions around the origin $o_{1}$. We
construct all local solutions by Frobenius method for the hypergeometric
system following \cite{HKTY,HLY} (see Appendix \ref{sec:App-GKZ-mirror-sym}
for a brief summary). The basic object is the indicial ideal which
we can read off from (\ref{eq:PF-system-o1}).
\begin{prop}
\label{prop:Ind(D)}Define the indicial ideal of the Picard-Fuchs
system (\ref{eq:PF-system-o1}) by
\[
Ind(\mathcal{D})=\left\langle \begin{smallmatrix}(\theta_{1}+\theta_{2}-\theta_{4})(\theta_{1}+\theta_{3}-\theta_{4}),\;(\theta_{1}+\theta_{2}-\theta_{4})(\theta_{2}+\theta_{3}-\theta_{4}),\;(\theta_{1}+\theta_{3}-\theta_{4})(\theta_{2}+\theta_{3}-\theta_{4}),\\
(\theta_{2}-\theta_{4})(\theta_{3}-\theta_{4}),\;(\theta_{1}-\theta_{4})(\theta_{3}-\theta_{4}),\;(\theta_{1}-\theta_{4})(\theta_{2}-\theta_{4}),\\
(\theta_{1}+\theta_{2}-\theta_{4})(\theta_{3}-\theta_{4}),\;(\theta_{1}+\theta_{3}-\theta_{4})(\theta_{2}-\theta_{4}),\;(\theta_{2}+\theta_{3}-\theta_{4})(\theta_{1}-\theta_{4})
\end{smallmatrix}\right\rangle .
\]
This is a zero dimensional ideal in $\mathbb{Q}[\theta_{1},\theta_{2},\theta_{3},\theta_{4}]$
where $\theta_{k}:=z_{k}\frac{\partial\;}{\partial z_{k}}$. \end{prop}
\begin{proof}
We verify the claimed property by calculating the Gr\"obner basis
of $Ind(\mathcal{D})$. 
\end{proof}
The fact that $Ind(\mathcal{D})$ is a zero dimensional ideal is one
of the properties for the origin $o_{1}$ to be a LCSL. The Picard-Fuchs
system (\ref{eq:PF-system-o1}) has further properties which are common
for the GKZ systems arising from mirror symmetry. 
\begin{prop}
\label{prop:Ind-ring-Mij}The quotient ring $\mathbb{Q}[\theta_{1},\cdots,\theta_{4}]/Ind(\mathcal{D})$
is of dimension 6, with its standard monomials $1;\theta_{1},\cdots,\theta_{4};\theta_{4}^{2}$.
The intersection pairing $M_{ij}=\langle\theta_{i}\theta_{j}\rangle$
(see Appendix \ref{sec:Appendix-proof}) is given by 
\begin{equation}
\left(M_{ij}\right)=d\times\left(\begin{matrix}0 & 1 & 1 & 1\\
1 & 0 & 1 & 1\\
1 & 1 & 0 & 1\\
1 & 1 & 1 & 1
\end{matrix}\right),\label{eq:Mat-Mij}
\end{equation}
where $d:=\langle\theta_{4}^{2}\rangle$ will be fixed (to be 2) later. \end{prop}
\begin{proof}
Since the indicial ideal is homogeneous, we have homogeneous basis
for the quotient. We can determine the standard monomials by making
Gr\"obner basis. The pairing $M_{ij}$ follows form the definition
of the $\mathbb{Q}$-linear map $\left\langle -\right\rangle :\mathbb{Q}[\theta]/Ind(\mathcal{D})\to\mathbb{Q}$
described in Appendix \ref{sec:Appendix-proof}.
\end{proof}
The following proposition is the content of the Frobenius method for
hypergeometric series of multi-varibles.
\begin{prop}
\label{prop:Frobenius-local-sol} (1) For the coefficient $c(n)$
in (\ref{eq:def-cn}), the following limits exist for all $n=(n_{1,}n_{2},n_{3},n_{4})\in\mathbb{Z}^{4}$:
\[
\lim_{\rho\to0}\frac{\partial\;}{\partial\rho_{i}}c(n+\rho),\;\;\lim_{\rho\to0}\sum_{i,j=1}^{4}M_{ij}\frac{\partial^{2}\;}{\partial\rho_{i}\partial\rho_{j}}c(n+\rho).
\]
In particular, these are non-vanishing only for $(n_{1,}n_{2},n_{3},n_{4})\in\mathbb{Z}_{\geq0}^{4}$.

\noindent (2) The complete set of solutions of the Picard-Fuchs system
(\ref{eq:PF-system-o1}) is given by 
\[
\omega_{0}(z),\;\omega_{i}^{(1)}(z):=\frac{\partial\;}{\partial\rho_{i}}\omega(z,\rho)\vert_{\rho=0},\;\omega^{(2)}(z):=-\frac{1}{2}\sum_{i,j}M_{ij}\frac{\partial\;}{\partial\rho_{i}}\frac{\partial\;}{\partial\rho_{j}}\omega_{0}(z,\rho)\vert_{\rho=0}
\]
 where $\omega(z,\rho):=\sum c(n+\rho)z^{n+\rho}$ and $\omega_{0}(z):=\omega(z,0)$.\end{prop}
\begin{proof}
The claims follow by applying the generalized Frobenius method described
in \cite{HKTY,HLY}. To avoid going into technical details, we defer
the proofs to Appendix \ref{sec:Appendix-proof}.\end{proof}
\begin{rem}
\label{rem:Remark-GKZ-Frob}The solutions in (2) indicate that the
classical Frobenius method for hypergeometric series of one variables
naively extends to hypergeometric series of multi-varibles. This was
the \textit{non-trivial} observation first made in \cite{HKTY,HLY}
for GKZ hypergeometric systems arising from the mirror symmetry. In
fact, it is easy to see that the limit $\lim_{\rho\to0}\frac{\partial^{2}\;}{\partial\rho_{i}\partial\rho_{j}}c(n+\rho)$
has non-vanishing contributions even when some of $n_{k}$'s are negative.
However, after summing up with $M_{ij}$, these contributions cancel
out and we obtain the power series solution $\sum_{i,j}M_{ij}\frac{\partial\;}{\partial\rho_{i}}\frac{\partial\;}{\partial\rho_{j}}\omega_{0}(z,\rho)\vert_{\rho=0}$.
In Appendix \ref{sec:Appendix-proof}, we will show an example where
a naive application of the Frobenius method for hypergeometric series
of multi-variables generates local solutions but in the form of Laurent
series. 
\end{rem}
\para{Transcendental lattice from the period relation} Recall that
generic members of our family of K3 surfaces have transcendental lattice
(\ref{eq:Tx}). We can read off this transcendental lattice by finding
a quadratic relation (period relation) satisfied by period integrals.
Let us first look at the symmetric form on $\oplus_{i}\mathbb{Z}\theta_{i}$
defined by matrix $(M_{ij})$ above (assuming $d$ is an integer). 
\begin{lem}
\label{lem:int-form-M-P}The lattice $\oplus_{i}\mathbb{Z}\theta_{i}$
with the symmetric form $(M_{ij})$ is isomorphic to $U(d)\oplus\left\langle -d\right\rangle \oplus\left\langle -d\right\rangle $,
i.e., we have 
\[
(M_{ij})=^{t}P\left(\begin{matrix}0 & d & 0 & 0\\
d & 0 & 0 & 0\\
0 & 0 & -d & 0\\
0 & 0 & 0 & -d
\end{matrix}\right)P\text{ with }P=\left(\begin{matrix}1 & 0 & 1 & 1\\
0 & 1 & 1 & 1\\
0 & 0 & 1 & 1\\
0 & 0 & 1 & 0
\end{matrix}\right).
\]
\end{lem}
\begin{proof}
It is easy to verify that the unimodular matrix $P$ gives the isomorphism. \end{proof}
\begin{prop}
\label{prop:period-relation-Mij}The following quadratic relation
holds:
\begin{equation}
\left(2\omega^{(2)}+d\pi^{2}\omega_{0}\right)\omega_{0}+\sum_{i,j}M_{ij}\omega_{i}^{(1)}\omega_{j}^{(1)}=0.\label{eq:quadratic-rel-sols}
\end{equation}
\end{prop}
\begin{proof}
We verify this by series expansions of the solutions to some higher
orders. 
\end{proof}
When $d=2$, using Lemma \ref{lem:int-form-M-P}, we find that the
above quadratic form and the form of the transcendental lattice $T_{X}=U(2)\oplus U(2)\oplus A_{1}\oplus A_{1}$
are consistent if the following conjecture holds: 
\begin{conjecture}
Period integrals 
\begin{equation}
\Pi(z)=\,^{t}\left(\omega_{0},\;\frac{2}{(2\pi i)^{2}}(\omega^{(2)}+\pi^{2}\omega_{0}),\;\frac{2}{2\pi i}\omega_{1}^{(1)},\cdots,\frac{2}{2\pi i}\omega_{4}^{(1)}\right)\label{eq:intPi}
\end{equation}
 are integral basis, i.e., have integral monodromy which preserve
the symmetric form $\left(\begin{smallmatrix}0 & 2\\
2 & 0
\end{smallmatrix}\right)\oplus\left(M_{ij}\right)$ with $d=2$. 
\end{conjecture}
Note that in this conjecture, we have introduced the factors $2\pi i$
to have integral local monodromy around the divisors $z_{i}=0$. This
integral structure of period integrals is in accord with the general
formulas which come from mirror symmetry, see Appendix \ref{sec:App-GKZ-mirror-sym}.
In Section \ref{sec:Mirror-symmetr}, we describe the mirror geometry
of $X$ using Proposition \ref{prop:period-relation-Mij}.

~

~

\section{\textbf{\textup{The $\lambda_{K3}$ functions from the master equation
\label{sec:Solving-Master-Eqs}}}}

We solve the master equation by determining the so-called mirror maps
around the boundary point $o_{1}$. Following the preceding section,
we will first make a local analysis around the fixed boundary point
$o_{1}$. Using the transformation property of the master equation,
we will finally arrive at the global expressions (\ref{eq:intro-lambdaK3})
and (\ref{eq:intro-lambdaK3plus}) for the $\lambda_{K3}$ functions.

\subsection{Mirror maps }

With the local solutions in Proposition \ref{prop:Frobenius-local-sol},
we can now define the mirror map locally around the point $o_{1}$.
\begin{defn}
\label{def:mirror-maps}(1) For the period integrals $\Pi(z)$ in
(\ref{eq:intPi}), we define 
\begin{equation}
t_{k}:=\frac{2}{2\pi i}\,\frac{\omega_{k}^{(1)}(z)}{\omega_{0}(z)}=\frac{2}{2\pi i}\log c_{k}z_{k}+\cdots(k=1,...,4),\label{eq:tk=00003Dlog}
\end{equation}
where $c_{1}=c_{2}=c_{3}=4$ and $c_{4}=1$. 

\noindent (2) Putting $Q_{k}:=e^{\pi it_{k}}$, we write the inverse
relations of (\ref{eq:tk=00003Dlog}) by 
\[
z_{k}(Q)=c_{k}Q_{k}+\cdots(k=1,...,4)
\]
and call them the \textit{mirror map} around the boundary point $o_{1}$. 
\end{defn}
We set 
\[
F:=\frac{2}{(2\pi i)^{2}}\frac{1}{\omega_{0}(z)}\big(\omega^{(2)}(z)+\pi^{2}\omega_{0}(z)\big).
\]
Then it holds that $\frac{\Pi(z)}{\omega_{0}(z)}=\,^{t}(1,F,t_{1},...,t_{4})$
for the period integrals $\Pi(z)$ (\ref{eq:intPi}). The quadratic
relation (\ref{eq:quadratic-rel-sols}) with $d=2$ becomes 
\[
4F+\sum_{i,j}M_{ij}t_{i}t_{j}=4F+4{\tt t}_{1}{\tt t}_{2}-2{\tt t}_{3}^{2}-2{\tt t}_{4}^{2}=0,
\]
where we define $(\mathtt{t}_{1},{\tt t}_{2},{\tt t}_{3},{\tt t}_{4}):=(t_{1},t_{2},t_{3},t_{4})\,^{t}P$
using the unimodular matrix $P$ in Lemma \ref{lem:int-form-M-P}.
This indicates that $[1,F,{\tt t}_{1},...,{\tt t}_{4}]$ is a point
in the period domain $\mathcal{D}_{K3}$ defined for the transcendental
lattice $T_{X}\simeq U(2)^{\oplus2}\oplus A_{1}^{\oplus2}$, namely
the map $z\mapsto[1,F,{\tt t}_{1},...,{\tt t}_{4}]$ composed with
the following isomorphism $\mu:\mathcal{D}_{K3}\simeq\mathbb{H}_{2}$
expresses the period map $\mathcal{P}:\mathcal{M}_{6}\to\mathbb{H}_{2}$
locally near the boundary point $o_{1}$. 
\begin{lem}
There is an isomorphism $\mu:\mathcal{D}_{K3}\simeq\mathbb{H}_{2}$
whose inverse $\mu^{-1}:\mathbb{H}_{2}\simeq\mathcal{D}_{K3}$ is
explicitly given by 
\[
\mu^{-1}(W)=\left[1,-\det W,w_{11},w_{22},\frac{w_{12}-iw_{21}}{1-i},\frac{w_{12}+iw_{21}}{1+i}\right]
\]
for $W=\left(\begin{matrix}w_{11} & w_{12}\\
w_{21} & w_{22}
\end{matrix}\right)$.\end{lem}
\begin{proof}
It is easy to verify the quadratic relation $\mu^{-1}(W).\mu^{-1}(W)=0$.
We refer \cite[Sect.\,1.3]{Matsu93} for the more details of the isomorphism. \end{proof}
\begin{defn}
\label{def:PA-E36-o1}Near the boundary point $o_{1}$, we describe
the period map $\mathcal{P}:\mathcal{M}_{6}\to\mathbb{H}_{2}$ by
\[
\mathcal{P}([A])=\mu\left(\left[1,F,{\tt t_{1},{\tt t}_{2},{\tt t}_{3},{\tt t}_{4}}\right]\right),
\]
where $\,^{t}(1,F,{\tt t_{1},{\tt t}_{2},{\tt t}_{3},{\tt t}_{4}})=\left(\begin{smallmatrix}1\\
 & 1\\
 &  & P
\end{smallmatrix}\right)\frac{\Pi(z)}{\omega_{0}(z)}$ is the period integral (\ref{eq:intPi}) determined near the boundary
point $o_{1}$ (hence $z=z(A))$ and $P$ is the unimodular matrix
in Lemma \ref{lem:int-form-M-P}. \end{defn}
\begin{rem}
Parallel to the above definition, we have local descriptions of the
period maps near each of $o_{1},o_{2};o_{1}^{+},o_{2}^{+},o_{3}^{+}$.
Because of $S_{6}$ invariance of the resolutions $\widetilde{\mathcal{M}}_{6}$
and $\widetilde{\mathcal{M}}_{6}^{+}$, we have local descriptions
of the period maps near all of the boundary points in the resolutions
as well.\textcolor{red}{{} }
\end{rem}
When we study the modular properties, the coordinate ${\tt t}_{k}$
introduced above is preferred to the coordinate $t_{k}$. Correspondingly
we define $q_{k}:=e^{\pi i{\tt t}_{k}}(k=1,2,3)$ and $q_{4}:=e^{\pi i(\mathtt{t}_{4}+1)}$
which are related to the $Q_{k}$ by 
\[
q_{1}=Q_{1}Q_{3}Q_{4},\,\,q_{2}=Q_{2}Q_{3}Q_{4},\,\,q_{3}=Q_{3}Q_{4},\,\,q_{4}=-Q_{3}
\]
or by 
\begin{equation}
Q_{1}=\frac{q_{1}}{q_{3}},\,\,Q_{2}=\frac{q_{2}}{q_{3}},\,\,Q_{3}=-q_{4},\,\,Q_{4}=-\frac{q_{3}}{q_{4}}.\label{eq:Qby-q}
\end{equation}

\begin{rem}
Here, the slightly mysterious shift in the definition $q_{4}=e^{\pi i(\mathtt{t}_{4}+1)}$
corresponds to changing the branch cut for the logarithms of $\log z_{3}$
and $\log z_{4}$, i.e., changing $\log z_{3}\to\log z_{3}+\pi i$
and $\log z_{4}\to\log z_{4}+\pi i$ in Definition \ref{def:mirror-maps}.
We \textit{do not} have a good understanding about this shift, but
this is necessary to express the mirror maps in terms of theta functions. 
\end{rem}
For convenience, we introduce the following notation. 
\begin{defn}
By $z_{k}(q)$ we represent the mirror map $z_{k}(Q)$ substituted
the relation (\ref{eq:Qby-q}), i.e., $z_{k}(q):=z_{k}(Q)\vert_{Q_{i}=Q_{i}(q)}$. \end{defn}
\begin{prop}
When we substitute the mirror map formally into the unique $($up
to constant$)$ power series $\omega_{0}(z)$, we have
\[
\begin{alignedat}{3}\omega_{0}(z(q))^{2} &  & = &  &  & \,1+8(q_{1}+q_{2})+24(q_{1}^{2}+q_{2}^{2}+4q_{1}q_{2})\\
 &  & - &  &  & \,8q_{1}q_{2}\Big\{4\big(\frac{1}{q_{3}}+q_{3}\big)-4\big(\frac{1}{q_{4}}+q_{4}\big)+\big(\frac{1}{q_{3}}+q_{3}\big)\big(\frac{1}{q_{4}}+q_{4}\big)\Big\}+\cdots.
\end{alignedat}
\]

\end{prop}
The above expression plays a role when studying the master equation.

\subsection{Solving the master equation 1\label{sub:masterEq-1}}

Now we can set up the master equation around the boundary point $o_{1}$
by using the period map $\mathcal{P}([A])$ given in Definition \ref{def:PA-E36-o1}
as follows:
\[
P_{I}(z)\omega_{0}(z)^{2}=\Theta^{\tau}\left(\begin{matrix}i\,j\,k\\
l\,m\,n
\end{matrix}\right)^{2}(\mathcal{P}([A]))\quad(I=\{\{i,j,k\},\{l,m,n\}\}),
\]
where $\omega_{0}(z)$ is the unique power series solution near $o_{1}$.
Note that both sides of this equation are given by $q$-expansions
when we substitute the mirror map $z_{k}=z_{k}(q)$. By explicit calculations,
we find the following property:
\begin{prop}
\label{prop:fixing-tau}When expanded into q-series, the master equation
holds to some higher order in $q$ only if we take $\tau=\left(\begin{smallmatrix}1\,2\,3\,4\,5\,6\\
3\,2\,6\,1\,5\,4
\end{smallmatrix}\right).$ 
\end{prop}
Now we can determine $z_{k}$(q) and $\omega_{0}(z(q))$ in terms
of the theta functions by solving the master equation 
\begin{equation}
P_{I}(z)\omega_{0}(z)^{2}=\Theta\left(\begin{matrix}\tau(i)\,\tau(j)\,\tau(k)\\
\tau(l)\,\tau(m)\,\tau(n)
\end{matrix}\right)^{2}\quad(I=\{\{i,j,k\},\{l,m,n\}\})\label{eq:master-eq-o1}
\end{equation}
for $z_{k}$, $\omega_{0}$. This is an overdetermined algebraic system.
However after some algebras, we find the following 
\begin{prop}
\label{prop:Sol-algebraic-master-eq}The above master equation has
a $($unique$)$ solution, 
\begin{equation}
\begin{alignedat}{5}z_{1} & = &  & \frac{\Theta_{3}^{2}+\Theta_{9}^{2}-\omega_{0}^{2}}{\omega_{0}^{2}-\Theta_{7}^{2}} & ,\quad & z_{2} & = &  & \frac{\Theta_{3}^{2}+\Theta_{9}^{2}-\omega_{0}^{2}}{\omega_{0}^{2}-\Theta_{9}^{2}} & ,\\
z_{3} & = &  & \frac{(\omega_{0}^{2}-\Theta_{7}^{2})(\omega_{0}^{2}-\Theta_{9}^{2})}{\omega_{0}^{2}(\Theta_{4}^{2}+\Theta_{9}^{2}-\omega_{0}^{2})} & ,\quad & z_{4} & = &  & \frac{\Theta_{4}^{2}+\Theta_{9}^{2}-\omega_{0}^{2}}{\Theta_{3}^{2}+\Theta_{9}^{2}-\omega_{0}^{2}},
\end{alignedat}
\label{eq:zk-theta}
\end{equation}

\end{prop}
\begin{equation}
\omega_{0}^{2}=\frac{1}{2\Theta_{8}^{2}}\left\{ \Theta_{7}^{2}\Theta_{8}^{2}-\Theta_{10}^{2}\Theta_{5}^{2}+\Theta_{6}^{2}\Theta_{9}^{2}-\widetilde{\Theta}\right\} ,\label{eq:w0-theta}
\end{equation}
where theta functions $\Theta_{i}$ and $\widetilde{\Theta}$ are
defined in Appendix \ref{sec:App-Theta-fn}. 
\begin{proof}
Since the master equation (\ref{eq:master-eq-o1}) (for $z_{k}$ and
$\omega_{0}^{2}$) is overdetermined, we select four equations to
solve for $z_{1},...,z_{4}$. For example, we can take four equations
indexed by the following $I=\{\{i,j,k\},\{l,m,n\}\}:$ 
\[
\{\{1,2,5\},\{3,4,6\}\},\{\{1,3,4\},\{2,5,6\}\},\{\{1,3,6\},\{2,4,5\}\},\{\{1,5,6\},\{2,3,4\}\}.
\]
Using the corresponding polynomials $P_{s}(z)$ $(s=2,3,7,9)$ in
Appendix \ref{sec:App-P-Q}, we obtain the claimed expressions (\ref{eq:zk-theta}).
Substituting these into the remaining six equations, it turns out
that these are equivalent to five linear relations among the theta
functions \cite[Rem.\,3.1.2]{Matsu93} and one additional equation;
\[
\Theta_{8}^{2}\omega_{0}^{4}-\left\{ \Theta_{7}^{2}\Theta_{8}^{2}-\Theta_{10}^{2}\Theta_{5}^{2}+\Theta_{6}^{2}\Theta_{9}^{2}\right\} \omega_{0}^{2}+\Theta_{6}^{2}\Theta_{7}^{2}\Theta_{9}^{2}=0.
\]
We solve this equation for $\omega_{0}^{2}$ to obtain 
\[
\omega_{0}^{2}=\frac{1}{2\Theta_{8}^{2}}\Big\{ T\pm\sqrt{T^{2}-4\Theta_{6}^{2}\Theta_{7}^{2}\Theta_{8}^{2}\Theta_{9}^{2}}\Big\},
\]
where $T:=\Theta_{7}^{2}\Theta_{8}^{2}-\Theta_{10}^{2}\Theta_{5}^{2}+\Theta_{6}^{2}\Theta_{9}^{2}$.
Using the linear relations, we can verify the following equality:
\[
T^{2}-4\Theta_{6}^{2}\Theta_{7}^{2}\Theta_{8}^{2}\Theta_{9}^{2}=\frac{1}{12}\Big\{\big(\sum_{k=1}^{10}\Theta_{k}^{4}\big)^{2}-4\sum_{k=1}^{10}\Theta_{k}^{8}\Big\}=\frac{2^{4}}{3^{2}5^{2}}\Theta^{2},
\]
where the second equality is nothing but the definition of the weight
four modular form $\Theta$ \cite[Prop.\,3.1.5]{Matsu93}. We finally
determine the sign of the square root so that the relation (\ref{eq:omega2})
below holds as the $q$-series.\end{proof}
\begin{rem}
We have solved the master equation (\ref{eq:master-eq-o1}) expressed
in the coordinate around the boundary point $o_{1}$. By the same
arguments given in the proof of Proposition \ref{prop:solving-MS-E24},
we can transform the master equation (\ref{eq:master-eq-o1}) to other
charts which cover $\mathcal{M}_{6}$, and see that these are equivalent
to (\ref{eq:master-eq-o1}). For this argument, we use the covering
property (\ref{eq:M6-sigma-Union}) of $\mathcal{M}_{6}$, the transformation
property (\ref{eq:Y-PI-E36}) and also the relation 
\[
\Theta\left(\begin{matrix}i\,j\,k\\
l\,m\,n
\end{matrix}\right)(g_{\sigma}\cdot W)=\vert CW+D\vert^{2}\,\Theta\left(\begin{matrix}\sigma(i)\,\sigma(j)\,\sigma(k)\\
\sigma(l)\,\sigma(m)\,\sigma(n)
\end{matrix}\right)(W)
\]
for $g_{\sigma}=\left(\begin{smallmatrix}A & B\\
C & D
\end{smallmatrix}\right)\in\Gamma_{T}$ which corresponds to $\sigma\in S_{6}$, see \cite[Sect.3.1]{Matsu93}. \end{rem}
\begin{prop}
\label{prop:w0-Sq-lambda-theta}Define the $\lambda_{K3}$ functions
by $\lambda_{k}=z_{k}(q)$ with (\ref{eq:zk-theta}) and (\ref{eq:w0-theta}).
Then for the hypergeometric series $\omega_{0}(z_{1},z_{2},z_{3},z_{4})$
in Proposition \ref{prop:PF-o1}, the following equality holds:
\begin{equation}
\omega_{0}(\lambda_{1},\lambda_{2},\lambda_{3},\lambda_{4})^{2}=\frac{1}{2\Theta_{8}^{2}}\left\{ \Theta_{7}^{2}\Theta_{8}^{2}-\Theta_{10}^{2}\Theta_{5}^{2}+\Theta_{6}^{2}\Theta_{9}^{2}-\widetilde{\Theta}\right\} \label{eq:omega2}
\end{equation}
where $\Theta_{i}=\Theta_{i}(q)$ and $\widetilde{\Theta}=\widetilde{\Theta}(q)$. 
\end{prop}
Clearly, the above relation is a generalization of the formula $\omega_{0}(\lambda(q))^{2}=\theta_{3}(q)^{2}$
(\ref{eq:elliptic-lambda}) which is classically known for the Legendre
family. Several direct proofs are known for the relation $\omega_{0}(\lambda(q))^{2}=\theta_{3}(q)^{2}$
in terms of Picard-Fuchs equations \cite[Sect.5.4]{Zagier}. We do
expect a similar direct proof for the above relation.  
\begin{rem}
\label{rem:o1-o2}The above analysis has been done starting from the
local solutions around $o_{1}$. We may also take the other boundary
point $o_{2}$, which are related by (Laurent) monomial relation (\ref{eq:z-o2-by-o1}).
It is easy to see that the Picard-Fuchs system (\ref{eq:PF-system-o1})
preserves the same form when we substitute the monomial relation (\ref{eq:z-o2-by-o1}).
Hence we obtain the same local solutions as in Propositions \ref{prop:PF-o1}
and \ref{prop:Frobenius-local-sol}; and the same calculations as
above apply to $o_{2}$, and in particular, we result in the same
form of the mirror map (\ref{eq:zk-theta}). However these two mirror
maps have different boundary conditions; the former vanishes at the
normal crossing divisors $z_{1}(o_{1})z_{2}(o_{1})z_{3}(o_{1})z_{4}(o_{1})=0$,
while the latter vanishes at $z_{1}(o_{2})z_{2}(o_{2})z_{3}(o_{2})z_{4}(o_{2})=0$.
We regard these mirror maps as different representations of a $\lambda_{K3}$-function,
which correspond to different forms of the elliptic $\lambda$-function,
cf.~(\ref{eq:table-zSigma}), with different vanishing conditions
at the cusps. 
\end{rem}

\subsection{Solving the master equation 2. \label{sub:masterEq-2}}

Completely parallel calculations apply to the flipped resolution $\widetilde{\mathcal{X}}_{e}^{+}\to\mathcal{X}_{e}$
where we found three boundary points $o_{i}^{+}(i=1,2,3)$. We denote
by $z_{k}(o_{i}^{+})$ $(i=1,2,3)$ the corresponding local coordinate
and set $\tilde{z}_{k}:=z_{k}(o_{1}^{+})$. From the definitions $z_{k}=\prod_{i}{\tt a}_{i}^{\ell_{i}^{(k)}}$
and $\tilde{z}_{k}=\prod_{i}\mathtt{a}^{\tilde{\ell}_{i}^{(k)}}$
(see \cite[Def.3.5]{HLTYpartI}) and the relations (\ref{eq:tilde-ell-by-ell}),
we see that the coordinate $\tilde{z}_{k}$ is related to the coordinate
$z_{k}=z_{k}(o_{1})$ of the other resolution $\widetilde{\mathcal{X}}_{e}\to\mathcal{X}_{e}$
by 
\[
\tilde{z}_{1}=z_{1},\,\,\tilde{z}_{2}=z_{1}z_{4},\,\,\tilde{z}_{3}=\frac{z_{2}}{z_{1}},\,\,\tilde{z}_{4}=\frac{z_{3}}{z_{1}}.
\]
Inverting this (Laurent) monomial relation as $z_{k}=z_{k}(\tilde{z})$,
we substitute into the polynomials $P_{I}(z)$. We directly check
the results are polynomials in $\tilde{z}_{k}$.
\begin{defn}
We define by $Q_{I}(\tilde{z})$ the polynomial $P_{I}(z)\vert_{z=z(\tilde{z})}$. 
\end{defn}
By definition, we have 
\[
\frac{1}{a_{0}b_{0}c_{0}}Y(A_{0})=P_{I}(z)=Q_{I}(\tilde{z}).
\]
It should be noted that $P_{I}(t)$ and $Q_{I}(t)$ are polynomials
of different shapes. 

In Appendix \ref{sec:App-Theta-fn}, we list the Picard-Fuchs system
in the coordinate $\widetilde{z}_{k}$. The origin $o_{1}^{+}$ of
this system is a LCSL where we have unique (up to constant) regular
solution $\omega_{0}(\tilde{z})$ and all others contain powers of
logarithms, $\log\tilde{z}_{k}$. The Frobenius method applies to
this case as well. By finding the quadratic relation satisfied by
solutions, we can describe the period map $\mathcal{P}:\mathcal{M}_{6}\to\mathbb{H}_{2}$
locally around $o_{1}^{+}$. This time we set up the master equation
in the following from
\begin{equation}
Q_{I}(\tilde{z})\omega_{0}(\tilde{z})^{2}=\Theta^{\rho}\left(\begin{matrix}i\,j\,k\\
l\,m\,n
\end{matrix}\right)^{2}(\mathcal{P}([A]))\quad(I=\{\{i,j,k\},\{l,m,n\}\}).\label{eq:master-eq-o1plus}
\end{equation}

\begin{prop}
\label{prop:fixing-rho}When (\ref{eq:master-eq-o1plus}) is expanded
in $q$-series, the master equation holds in lower degrees in $q$
only when when we take $\rho=\left(\begin{smallmatrix}1\,2\,3\,4\,5\,6\\
1\,4\,5\,3\,6\,2
\end{smallmatrix}\right).$ 
\end{prop}
Using the above element $\rho\in S_{6}$, we set up the algebraic
master equation for $\tilde{z}_{k}$ and $\omega_{0}$. Corresponding
to Proposition \ref{prop:Sol-algebraic-master-eq} ans Proposition
\ref{prop:w0-Sq-lambda-theta}, we obtain 
\begin{prop}
\label{prop:mirrorMap-o1plus}The master equation has a unique solution,
\begin{equation}
\begin{alignedat}{5}\tilde{z}_{1} & = &  & \frac{\Theta_{3}^{2}+\Theta_{9}^{2}-\omega_{0}^{2}}{\omega_{0}^{2}-\Theta_{6}^{2}} & ,\quad & \tilde{z}_{2} & = &  &  & \frac{\Theta_{4}^{2}+\Theta_{9}^{2}-\omega_{0}^{2}}{\omega_{0}^{2}-\Theta_{6}^{2}},\\
\tilde{z}_{3} & = &  & \frac{\omega_{0}^{2}-\Theta_{6}^{2}}{\omega_{0}^{2}-\Theta_{9}^{2}}, & \quad & \tilde{z}_{4} & = &  &  & \frac{(\omega_{0}^{2}-\Theta_{6}^{2})^{2}(\omega_{0}^{2}-\Theta_{9}^{2})}{\omega_{0}^{2}(\Theta_{3}^{2}+\Theta_{9}^{2}-\omega_{0}^{2})(\Theta_{4}^{2}+\Theta_{9}^{2}-\omega_{0}^{2})},
\end{alignedat}
\label{eq:zk-theta-plus}
\end{equation}

\end{prop}
\begin{equation}
\omega_{0}^{2}=\frac{1}{2\Theta_{8}^{2}}\left\{ \Theta_{7}^{2}\Theta_{8}^{2}-\Theta_{10}^{2}\Theta_{5}^{2}+\Theta_{6}^{2}\Theta_{9}^{2}-\widetilde{\Theta}\right\} ,\label{eq:w0-theta-plus}
\end{equation}
where $\Theta_{i}$, and $\widetilde{\Theta}$ are defined in Appendix
\ref{sec:App-Theta-fn}. 
\begin{prop}
\label{prop:omega0-o1plus}Define the $\lambda_{K3}^{+}$ functions
by $\lambda_{k}^{+}=\tilde{z}_{k}(q)$ with (\ref{eq:zk-theta}) and
(\ref{eq:w0-theta}). Then, for the hypergeometric series $\omega_{0}(\tilde{z}_{1},\tilde{z}_{2},\tilde{z}_{3},\tilde{z}_{4})$
in Appendix \ref{sec:App.PF-equation-Bo1}, the following equality
holds:
\begin{equation}
\omega_{0}(\lambda_{1}^{+},\lambda_{2}^{+},\lambda_{3}^{+},\lambda_{4}^{+})^{2}=\frac{1}{2\Theta_{8}^{2}}\left\{ \Theta_{7}^{2}\Theta_{8}^{2}-\Theta_{10}^{2}\Theta_{5}^{2}+\Theta_{6}^{2}\Theta_{9}^{2}-\widetilde{\Theta}\right\} \label{eq:omega2-plus}
\end{equation}
where $\Theta_{i}=\Theta_{i}(q)$ and $\widetilde{\Theta}=\widetilde{\Theta}(q)$. 
\end{prop}
\par
\begin{rem}
As above, we arrived at the two definitions of the K3 analogues of
elliptic lambda functions, $\lambda_{k}$ and $\lambda_{k}^{+}$corresponding
to the resolutions $\widetilde{\mathcal{X}}_{e}\to\mathcal{X}_{e}$
and $\widetilde{\mathcal{X}}_{e}^{+}\to\mathcal{X}_{e}$, respectively.
As described in Remark \ref{rem:o1-o2}, these two have different
behavior near the normal crossing boundary divisors in the different
resolutions $\widetilde{\mathcal{M}}_{6}$ and $\widetilde{\mathcal{M}}_{6}^{+}$.
We regards these $\lambda_{k}$ and $\lambda_{k}^{+}$ are non-isomorphic
since these are defined on the non-isomorphic resolutions. 

~

~
\end{rem}

\section{\textbf{\textup{Mirror symmetry to a double cover of $Bl_{3}\mathbb{P}^{2}$
\label{sec:Mirror-symmetr}}}}

We can read off a mirror correspondence of the K3 surfaces $X$ from
the period integrals near the LCSLs (see Appendix \ref{sec:App-GKZ-mirror-sym}).
Extending general observations made in \cite{HKTY,HLY,HLY2} to the
present case, we identify the mirror partner of $X$ starting from
inspecting the structure of the ring defined by the indicial ideal
$Ind(D)$.

\subsection{A double cover of $Bl_{3}\mathbb{P}^{2}$ }

Let $Bl_{3}\mathbb{P}^{2}$ be a blow-up at three (general) points
of $\mathbb{P}^{2}$, which is a del Pezzo surface $S_{6}$ of degree
6. We denote by $E_{1},E_{2},E_{3}$ the exceptional divisors and
by $H$ the pull-back of the hyperplane class in $\mathbb{P}^{2}$.
Then following the lemma is immediate:
\begin{lem}
Define $L_{i}=H-E_{i}\;(i=1,2,3)$ and $L_{4}=H$. Then the intersection
form is given by 
\begin{equation}
\left(L_{i}\cdot L_{j}\right)=\left(\begin{matrix}0 & 1 & 1 & 1\\
1 & 0 & 1 & 1\\
1 & 1 & 0 & 1\\
1 & 1 & 1 & 1
\end{matrix}\right).\label{eq:Lij}
\end{equation}

\end{lem}
We identify the above intersection form, up to the factor $d=2$,
with (\ref{eq:Mat-Mij}) appeared in Proposition \ref{prop:Ind-ring-Mij}.
We explain the factor $2$ by considering a double cover: Consider
two general elements $g_{i1},g_{i2}\in|H-E_{i}|$ for each one dimensional
linear system $|H-E_{i}|$ on $S_{6}$. We define the double cover
$\overline{S}_{6}\to S_{6}$ branched along the zero locus $\left\{ g_{11}g_{12}g_{21}g_{22}g_{31}g_{32}=0\right\} $. 
\begin{prop}
The double cover of $\overline{S}_{6}$ is a $K3$ surface which is
singular at 12 points of $A_{1}$ singularities. Its Picard lattice
is generated by the proper transforms $\tilde{L}_{i}$ of $L_{i}\;(i=1,..,4)$
with the intersection matrix $(\tilde{L}_{i}\cdot\tilde{L}_{j})=2\,(L_{i}\cdot L_{j})$.\end{prop}
\begin{proof}
The number of intersection points is immediate by counting intersection
numbers of the divisors $\left\{ g_{ia}=0\right\} .$ The intersection
forms $(\tilde{L}_{i}\cdot\tilde{L}_{j})$ are doubled by the double
covering.
\end{proof}
The intersection form of $\overline{S}_{6}$ explains the factor $d=2$
in Proposition \ref{prop:period-relation-Mij}. Let us recall that
the K3 surface $X$ is defined to be a resolution of the singular
double cover $\overline{X}\to\mathbb{P}^{2}$ branched along general
six lines. Based on the form of mirror symmetry observed for hypersurfaces
in toric varieties \cite{HKTY,Hos}, we conjecture the following (cf.
the next section):
\begin{conjecture}
\label{conj:Mirror-D-C}Mirror of the double cover (singular) K3 surface
$\overline{X}$ is a singular K3 surface $\overline{S}_{6}$ defined
above. Namely, the double covering of del Pezzo surface $S_{6}$ branched
along the zero loci of general elements $g_{ia}\in|H-E_{i}|$ (i=1,2,3;
a=1,2).
\end{conjecture}
The relation of the conjecture to the standard descriptions of mirror
symmetry of K3 surfaces \cite{Batyrev,DoMs,GW} is not completely
clear, since the lattice $U(2)$ instead of $U$ is contained as a
summand of the transcendental lattice $T_{\overline{X}}$, for example.
However, we interpret below the conjecture as a variant of the so-called
Batyrev-Borisov toric mirror construction. 

\def\DNabla{
\begin{xy}
(-20,0)*{\Delta},
(60,0)*{\nabla},
(-13,-17)*{\Delta_1},
(10,-14)*{\Delta_3},
(5,-30)*{\Delta_2},
(35,-13)*{\nabla_2},
(51,-20)*{\nabla_1},
(32,-28)*{\nabla_3},
(0,0)*\dir{*},    
(-5,10)*\dir{*},  
(-5,5)*\dir{*},   
(-5,0)*\dir{*},   
(-5,-5)*\dir{*},  
(0,-5)*\dir{*},   
(5,-5)*\dir{*},   
(10,-5)*\dir{*},  
(5,0)*\dir{*},    
(0,5)*\dir{*},    
(40,0)*\dir{*},   
(35,0)*\dir{*},   
(35,-5)*\dir{*},  
(40,-5)*\dir{*},  
(45,0)*\dir{*},   
(45,5)*\dir{*},   
(40,5)*\dir{*},   
(0,-20)*\dir{*},    
(-5,-10)*{\cdot},   
(-5,-15)*\dir{*},   
(-5,-20)*\dir{*},   
(-5,-25)*{\cdot},   
(0,-25)*\dir{*},    
(5,-25)*\dir{*},    
(10,-25)*{\cdot},   
(5,-20)*\dir{*},    
(0,-15)*\dir{*},    
(40,-20)*\dir{*},   
(35,-20)*{\cdot},   
(35,-25)*\dir{*},   
(40,-25)*{\cdot},   
(45,-20)*\dir{*},   
(45,-15)*{\cdot},   
(40,-15)*\dir{*},   
\ar@{-} (-5,-5);(-5,10)
\ar@{-} (-5,-5);(10,-5)
\ar@{-} (-5,10);(10,-5)
\ar@{-} (35,0);(35,-5)
\ar@{-} (35,-5);(40,-5)
\ar@{-} (40,-5);(45,0)
\ar@{-} (45,0);(45,5)
\ar@{-} (45,5);(40,5)
\ar@{-} (40,5);(35,0)
\ar@{-} (0,-20);(-5,-15)
\ar@{-} (0,-20);(-5,-20)
\ar@{-} (-5,-15);(-5,-20)
\ar@{-} (0,-20);(0,-25)
\ar@{-} (0,-20);(5,-25)
\ar@{-} (0,-25);(5,-25)
\ar@{-} (0,-20);(5,-20)
\ar@{-} (0,-20);(0,-15)
\ar@{-} (5,-20);(0,-15)
\ar@{-} (40,-20);(35,-25)
\ar@{-} (40,-20);(45,-20)
\ar@{-} (40,-20);(40,-15)
\end{xy}}\vskip0.2cm 
\begin{table}[htbp]
\[ \DNabla \]
\caption{{\bf Fig.2} Batyrev-Borisov duality for the Minkowski sums $\Delta=\Delta_1+\Delta_2+\Delta_3$ and $\nabla=\nabla_1+\nabla_2+\nabla_3$.}
\end{table}  \vskip-0.8cm \; 

~

\subsection{Double coverings from the Batyrev-Borisov duality}

It is suggestive to arrange the combinatorial data for the constructions
$\overline{X}$ and $\overline{S}_{6}$ into a generalization of Batyrev-Borisov
toric mirror construction \cite{Batyrev,Batyrev-Borisov}. In a follow
up paper \cite{HLLY}, we will provide a full generalization to all
dimensional Calabi-Yau varieties.

\para{Batyrev-Borisov duality} Recall that, in toric geometry, the
projective plane $\mathbb{P}^{2}$ is described as $\mathbb{P}_{\Delta}$
with a two dimensional polytope 
\[
\Delta={\rm Conv}\left\{ (2,-1),(-1,2),(-1,-1)\right\} ,
\]
whose integral points represent sections of $-K_{\mathbb{P}^{2}}$.
We consider the following Minkowski sum decomposition 
\begin{equation}
\Delta=\Delta_{1}+\Delta_{2}+\Delta_{3}\label{eq:Minkowski-decomp-D}
\end{equation}
with $\Delta_{1}={\rm Conv\left\{ (-1,0),(-1,1),(0,0)\right\} },\Delta_{2}={\rm Conv\left\{ (0,-1),(1,-1),(0,0)\right\} }$
and $\Delta_{3}={\rm Conv\left\{ (1,0),(0,1),(0,0)\right\} .}$ This
decomposition corresponds to the factorization $f_{\Delta}=f_{\Delta_{1}}f_{\Delta_{2}}f_{\Delta_{3}}$
of Laurent polynomials of $\Delta$ into polynomials defined for each
polytope $\Delta_{i}$. In terms of homogeneous coordinates, this
is nothing but a factorization of cubic polynomials into three linear
polynomials. 

According to Batyrev-Borisov construction, we define the following
polar dual
\[
\nabla:=\left({\rm Conv\left\{ \Delta_{1},\Delta_{2},\Delta_{3}\right\} }\right)^{*}.
\]
Then, the Minkowski decomposition (\ref{eq:Minkowski-decomp-D}) induces
the corresponding Minkowski decomposition of $\nabla$,
\[
\nabla=\nabla_{1}+\nabla_{2}+\nabla_{3}
\]
with $\nabla_{1}={\rm Conv\left\{ (1,0),(0,0)\right\} ,\nabla_{2}={\rm Conv}\left\{ (0,1),(0,0)\right\} }$
and $\nabla_{3}={\rm Conv}\{(-1,-1),$ $(0,0)\}$. By Batyrev-Borisov
duality, we obtain the original $\Delta$ by the polar dual 
\[
\Delta=\left({\rm Conv}\left\{ \nabla_{1},\nabla_{2},\nabla_{3}\right\} \right)^{*}.
\]
The duality holds in general for the so-called reflexive polytopes
with additional data called nef-partitions. In Fig.2, we summarize
the duality in the present case. It is clear that the toric variety
$\mathbb{P}_{\nabla}$ is isomorphic to $Bl_{3}\mathbb{P}^{2}$, the
blow-up at three coordinate points of $\mathbb{P}^{2}$. 

\para{Double coverings from the duality} In the Batyrev-Bosisov duality,
associated to the polytope $\Delta_{i}$ (respectively $\nabla_{j}$
), we have Laurent polynomial $f_{\Delta_{i}}$ ($g_{\nabla_{j}}$)
and also a toric divisor $D_{\Delta_{i}}$ in $\mathbb{P}_{\nabla}$
($D_{\nabla_{j}}$in $\mathbb{P}_{\Delta}$). They may be summarized
in 
\[
f_{\Delta_{i}}\in H^{0}(\mathbb{P}_{\Delta},\mathcal{O}(D_{\nabla_{i}}))\text{ and }g_{\nabla_{j}}\in H^{0}(\mathbb{P}_{\nabla},\mathcal{O}(D_{\Delta_{j}}))
\]
under the duality. 
\begin{defn}
Suppose two reflexive polytopes have Minkowski sum decompositions
$\Delta=\Delta_{1}+\cdots+\Delta_{s}$ and $\nabla=\nabla_{1}+\cdots+\nabla_{s}$
which are dual in the sense of Batyrev-Borisov. Take two general sections
$f_{\Delta_{i},1},f_{\Delta_{i},2}\in H^{0}(\mathbb{P}_{\Delta},\mathcal{O}(D_{\nabla_{i}}))$
and $g_{\nabla_{j},1},g_{\nabla_{j},2}\in H^{0}(\mathbb{P}_{\nabla},\mathcal{O}(D_{\Delta_{j}}))$
for each divisors $D_{\Delta_{i}}$ and $D_{\nabla_{j}}$. We define
the double covering $\overline{Y}_{\Delta}$ of toric Fano variety
$\mathbb{P}_{\Delta}$ branched along $\cup_{i,a}\left\{ f_{\Delta_{i},a}=0\right\} $,
and similarly $\overline{Y}_{\nabla}$ of $\mathbb{P}_{\nabla}$ with
the branch locus $\cup_{j,a}\left\{ g_{\nabla_{j},a}=0\right\} $
in $\mathbb{P}_{\nabla}$.
\end{defn}
By construction, the double covers $\overline{Y}_{\Delta}$ and $\overline{Y}_{\nabla}$
are Calabi-Yau varieties which is singular in general. Our observation
made in Conjecture \ref{conj:Mirror-D-C} can be understood as a special
case of the pair of double covers in dimensions two, i.e., $(\overline{Y}_{\Delta},\overline{Y}_{\nabla})=(\overline{X},\overline{S}_{6})$.
We naturally expect that these double cover Calabi-Yau varieites $\overline{Y}_{\Delta}$
and $\overline{Y}_{\nabla}$ are mirror symmetric in general as we
have observed in the special case. Other geometric justifications
(e.g. \cite{SYZ,GS1,GS2,DHT}) for this new duality are also expected,
but we defer them to future investigations. 

~

~

~

~

\newpage

\appendix
\renewcommand{\themyparagraph}{{\Alph{section}.\arabic{subsection}.\alph{myparagraph}}}

\section{\textbf{\textup{Genus two theta functions}} \label{sec:App-Theta-fn}}

Here we summarize our notation for the genus two theta functions following
\cite{Matsu93,DoOrt}. 
\begin{defn}
For $W\in\mathbb{H}_{2}$ and $a,b\in\frac{1}{2}\mathbb{Z}[i]^{2}$,
we define theta functions on $\mathbb{H}_{2}$ by 
\[
\Theta\left[\begin{matrix}a\\
b
\end{matrix}\right](W):=\sum_{n\in\mathbb{Z}[i]^{2}}\exp\left(\pi i\big(\,^{t}\overline{(n+a)}W(n+a)+2\mathrm{Re}(\,^{t}\overline{b}\,n)\big)\right).
\]

\end{defn}
The theta functions $\Theta\left(\begin{matrix}i\,j\,k\\
l\,m\,n
\end{matrix}\right)(W)$ used in the text are special types given by $a,b$ satisfying $\mathrm{Re}(a)=\mathrm{Im}(a),$$\mathrm{Re}(b)=\mathrm{Im}(b)$,
which are specified by the correspondence 

\[
\left[\begin{matrix}a\\
b
\end{matrix}\right]=\frac{1+i}{2}\left[\begin{smallmatrix}s_{1}\\
s_{2}\\
s_{3}\\
s_{4}
\end{smallmatrix}\right]\leftrightarrow\left(\begin{matrix}i\,j\,k\\
l\,m\,n
\end{matrix}\right).
\]
Explicitly, we use the following correspondences \cite{Matsu93}:
\[
\begin{aligned}\left[\begin{smallmatrix}1\\
1\\
1\\
1
\end{smallmatrix}\right]\leftrightarrow\left(\begin{matrix}1\,2\,3\\
4\,5\,6
\end{matrix}\right), &  & \left[\begin{smallmatrix}1\\
1\\
0\\
0
\end{smallmatrix}\right]\leftrightarrow\left(\begin{matrix}1\,2\,4\\
3\,5\,6
\end{matrix}\right), &  & \left[\begin{smallmatrix}1\\
0\\
0\\
0
\end{smallmatrix}\right]\leftrightarrow\left(\begin{matrix}1\,2\,5\\
3\,4\,6
\end{matrix}\right),\\
\left[\begin{smallmatrix}1\\
0\\
0\\
1
\end{smallmatrix}\right]\leftrightarrow\left(\begin{matrix}1\,2\,6\\
3\,4\,5
\end{matrix}\right), &  & \left[\begin{smallmatrix}0\\
1\\
0\\
0
\end{smallmatrix}\right]\leftrightarrow\left(\begin{matrix}1\,3\,4\\
2\,5\,6
\end{matrix}\right), &  & \left[\begin{smallmatrix}0\\
0\\
0\\
0
\end{smallmatrix}\right]\leftrightarrow\left(\begin{matrix}1\,3\,5\\
2\,4\,6
\end{matrix}\right), &  & \left[\begin{smallmatrix}0\\
0\\
0\\
1
\end{smallmatrix}\right]\leftrightarrow\left(\begin{matrix}1\,3\,6\\
2\,4\,5
\end{matrix}\right),\\
\left[\begin{smallmatrix}0\\
0\\
1\\
1
\end{smallmatrix}\right]\leftrightarrow\left(\begin{matrix}1\,4\,5\\
2\,3\,6
\end{matrix}\right), &  & \left[\begin{smallmatrix}0\\
0\\
1\\
0
\end{smallmatrix}\right]\leftrightarrow\left(\begin{matrix}1\,4\,6\\
2\,3\,5
\end{matrix}\right), &  & \left[\begin{smallmatrix}0\\
1\\
1\\
0
\end{smallmatrix}\right]\leftrightarrow\left(\begin{matrix}1\,5\,6\\
2\,3\,4
\end{matrix}\right).
\end{aligned}
\]
We also write these ten theta functions by $\Theta_{i}$ with $i=1,...,10$
by ordering the theta functions $\Theta\left(\begin{matrix}i\,j\,k\\
l\,m\,n
\end{matrix}\right)$ from the left to right, and the first line to the third line in the
above correspondence. These functions have $q$-series expansions
with 
\[
q_{1}=e^{\pi iw_{11}},\;q_{2}=e^{\pi iw_{22}},\;q_{3}q_{4}=e^{\pi i(w_{12}+w_{21})},\;\frac{q_{3}}{q_{4}}=e^{-\pi(w_{12}-w_{12})}
\]
for $W=\left(\begin{matrix}w_{11} & w_{12}\\
w_{21} & w_{22}
\end{matrix}\right)$. The following properties are known in literatures (see \cite[Prop.3.1.1, Cor. 3.2.2]{Matsu93}):
\begin{prop}
The squares of the ten theta functions $\Theta_{i}(W)^{2}$ are modular
forms on $\Gamma_{T}(1+i)$ with the character $\det:\Gamma_{T}(1+i)\to\mathbb{C}^{*}$.
Any five of linearly independent theta functions freely generate the
module $\mathrm{M}_{\det}(\Gamma_{T}(1+i))$. 
\end{prop}
There is also a modular form $\Theta(W)$ on $\Gamma_{T}$ with a
certain character $\chi:\Gamma_{T}\to\left\{ \pm1\right\} $ \cite[Lem. 3.1.3]{Matsu93},
which satisfies 
\[
\Theta(W)^{2}=\frac{3\cdot5^{2}}{2^{6}}\left\{ \left(\sum_{i=1}^{10}\Theta_{i}(W)^{4}\right)^{2}-4\sum_{i=1}^{10}\Theta_{i}(W)^{8}\right\} .
\]
This relation is implicitly contained in the (algebraic) master equation
as shown in Proposition \ref{prop:Sol-algebraic-master-eq}. When
expressing our lambda functions and the hypergeometric series $\omega_{0}(\lambda)^{2}$,
we have introduced $\widetilde{\Theta}(W):=\frac{2^{2}}{3\cdot5}\Theta(W)$. 

~

~

\section{\textbf{\textup{The polynomials $P_{I}$ and $Q_{I}$ \label{sec:App-P-Q}}}}

Recall that we have defined the polynomials $P_{I}(t)$ and $Q_{I}(t)$
by expressing the (inhomogeneous) semi-invariants $\frac{1}{a_{0}b_{0}c_{0}}Y_{I}(A_{0})$
by the affine coordinates $z_{k}:=z_{k}(o_{1})$ and $\tilde{z}_{k}:=z_{k}(o_{1}^{+})$,
respectively. By definition, we have 
\[
\frac{1}{a_{0}b_{0}c_{0}}Y_{I}(A_{0})=P_{I}(z)\text{ and }\frac{1}{a_{0}b_{0}c_{0}}Y_{I}(A_{0})=Q_{I}(\tilde{z}).
\]
The coding of the semi-invariants by $I=\left\{ \{ijk\},\{lmn\}\right\} $
comes from the definition $Y_{I}(A)=[ijk][lmn]$. We use the definitions
given in \cite[Appendix C]{HLTYpartI};
\[
\begin{aligned}Y_{0}=[1\,2\,3][4\,5\,6], &  & Y_{1}=[1\,2\,4][3\,5\,6], &  & Y_{2}=[1\,2\,5][3\,4\,6],\\
Y_{3}=[1\,3\,4][2\,5\,6], &  & Y_{4}=[1\,3\,5][2\,4\,6],\\
Y_{6}=[1\,2\,6][3\,4\,5], &  & Y_{7}=[1\,3\,6][2\,4\,5], &  & Y_{8}=[1\,4\,6][2\,3\,5],\\
Y_{9}=[1\,5\,6][2\,3\,4], &  & Y_{10}=[1\,4\,5][2\,3\,6],
\end{aligned}
\]
where $Y_{5}$ is used for the degree two element. We denote by $P_{s}$
and $Q_{s}$ the corresponding polynomials to the $Y_{s}$ above.
These numbering should not be confused with the numbering $\Theta_{i}(i=1,...,10)$.
Below we list the polynomials : 
\[
\begin{aligned}P_{0}= & 1-(z_{1}z_{2}+z_{1}z_{3}+z_{2}z_{3}+z_{1}z_{2}z_{3})z_{4}-z_{1}z_{2}z_{3}z_{4}^{2},\\
P_{1}= & -z_{1}z_{2}(1+z_{3})z_{4},\;\;\;\;\;\,P_{2}=-1+z_{1}z_{3}z_{4},\;\;\;\;\;\;\;\;\;\;\;\;P_{3}=-1+z_{2}z_{3}z_{4},\\
P_{4}= & -z_{1}z_{2}(1+z_{3}z_{4})z_{4},\;\;\,P_{6}=-(1+z_{1}z_{4})z_{2}z_{3}z_{4},\;\;\,P_{7}=z_{1}(1+z_{2})z_{3}z_{4},\\
P_{8}= & (1+z_{1})z_{2}z_{3}z_{4},\;\;\;\;\;\;\;\;\;\;P_{9}=-z_{1}z_{3}z_{4}(1+z_{2}z_{4}),\;\;P_{10}=1-z_{1}z_{2}z_{4}.
\end{aligned}
\]
The polynomial $Q_{k}(\tilde{z})$ is determined by $Q_{k}(\tilde{z})=P_{k}(z)$
by substituting the relations $z_{1}=\tilde{z}_{1},z_{2}=\tilde{z}_{1}\tilde{z_{3}}$,
$z_{3}=\tilde{z}_{2}\tilde{z}_{4},z_{4}=\tilde{z}_{2}/\tilde{z}_{1}$.
Note that $Q_{s}(\tilde{z})$ are polynomials in $\tilde{z}_{k}$
although $\tilde{z}_{1}$ appears in the denominator of $z_{4}$. 

~

~

\section{\textbf{\textup{The representation $z_{k}^{\sigma}=\varphi_{\sigma}(z_{1},z_{2},z_{3},z_{4})$
of $S_{6}$}} \label{sec:App-representation-S6}}

The right action of $\sigma\in S_{3}$ on $2\times4$ matrix $A$
defines an element of $\mathrm{Aut}(\mathcal{M}_{4})$, $z(A)\mapsto z^{\sigma}(A):=z(A\sigma)$.
This naturally gives rise to the well-known representation $S_{3}\to\mathrm{Aut}(\mathcal{M}_{4})\simeq\mathrm{Aut}(\mathbb{P}^{1})$:
\[
\begin{matrix}\sigma & : & e & (12) & (23) & (23)(12) & (12)(23) & (13)\\
z^{\sigma} & : & z & \frac{1}{z} & \frac{z}{z-1} & 1-\frac{1}{z} & \frac{1}{1-z} & 1-z\\
G(\sigma,e) & : & 1 & z & 1-z & -z & z-1 & -1
\end{matrix}
\]
Here we have included the twist factor $G(\sigma,e)$ defined in (\ref{def:twist-G-E24}).
In a similar way, we have the representation $S_{6}\to\mathrm{Aut}(\mathcal{M}_{6})$
induced by the right action of $\sigma\in S_{6}$ on $3\times6$ matrices.
This action naturally defines the corresponding transformation $z_{k}^{\sigma}=\varphi_{\sigma}(z_{1},z_{2},z_{3},z_{4})$
on the affine coordinates $z_{k}$ of the resolutions. For the case
$z_{k}=z_{k}(o_{1})$, we present explicit forms of the transformations
for some $\sigma\in S_{6}$. Although expressions become complicated
in general, these should be regarded as the generalization of the
rational transformations given in the above table. 

$\sigma=\left(\begin{smallmatrix}1\,2\,3\,4\,5\,6\\
1\,2\,3\,6\,5\,4
\end{smallmatrix}\right)$
\[
z_{1}^{\sigma}=\frac{1}{z_{1}},\;\;z_{2}^{\sigma}=-z_{2}z_{3}z_{4},\;\;z_{3}^{\sigma}=\frac{1}{z_{1}z_{4}},\;\;z_{4}^{\sigma}=\frac{z_{1}}{z_{3}};\;\;G(\sigma,e)=\frac{1}{z_{1}z_{3}z_{4}}
\]

$\sigma=\left(\begin{smallmatrix}1\,2\,3\,4\,5\,6\\
2\,3\,4\,5\,6\,1
\end{smallmatrix}\right)$
\[
\begin{aligned} & z_{1}^{\sigma}=-\frac{z_{2}}{1+z_{2}},\;\;z_{2}^{\sigma}=-\frac{1-z_{1}z_{3}z_{4}}{1+z_{1}z_{4}},\;\;z_{3}^{\sigma}=-\frac{1-z_{1}z_{3}z_{4}}{1+z_{3}z_{4}}\\
 & z_{4}^{\sigma}=-\frac{(1+z_{1}z_{4})(z+z_{3}z_{4})}{1-z_{1}z_{3}z_{4}};\;\;\;G(\sigma,e)=-\frac{1}{1+z_{2}}.
\end{aligned}
\]

$\sigma=\left(\begin{smallmatrix}1\,2\,3\,4\,5\,6\\
6\,5\,4\,3\,2\,1
\end{smallmatrix}\right)$
\[
\begin{aligned}z_{1}^{\sigma}=-\frac{(1+z_{1})z_{2}(1+z_{3})z_{4}}{(1+z_{2}z_{4})(1-z_{1}z_{3}z_{4})},\;\; & z_{2}^{\sigma}=-\frac{z_{1}(1+z_{2})(1+z_{3})z_{4}}{(1+z_{1}z_{4})(1-z_{2}z_{3}z_{4})},\\
z_{3}^{\sigma}=-\frac{(1+z_{1})(1+z_{2})z_{3}z_{4}}{(1-z_{1}z_{2}z_{4})(1+z_{3}z_{4})},\;\; & z_{4}^{\sigma}=\frac{(1+z_{1}z_{4})(1+z_{2}z_{4})(1+z_{3}z_{4})}{(1+z_{1})(1+z_{2})(1+z_{3})z_{4}}.
\end{aligned}
\]

\[
G(\sigma,e)=\frac{1-(z_{1}z_{2}+z_{1}z_{3}+z_{2}z_{3}+z_{1}z_{2}z_{3})z_{4}-z_{1}z_{2}z_{3}z_{4}^{2}}{(1-z_{1}z_{2}z_{4})(1-z_{1}z_{3}z_{4})(1-z_{2}z_{3}z_{4})}.
\]
We have similar expressions for the other boundary points $o_{i}$
and $o_{i}^{+}$ as well. 

~

~

\section{\textbf{\textup{Mirror symmetry and the generalized Frobenius method}}
\label{sec:App-GKZ-mirror-sym}}

Here we summarize briefly the Frobenius method formulated in \cite{HKTY,HLY,HLY2}
for the GKZ systems which determines period integrals of Calabi-Yau
complete intersections. The (generalized) Frobenius method applies
to the local solutions about special boundary points, i.e., LCSLs.
Assume that $X$ is a K3 surface given as complete intersection in
a toric variety, and $X^{*}$ is the mirror K3 surface determined
by Batyrev-Borisov toric mirror symmetry \cite{Batyrev,Batyrev-Borisov}.
In this setting, we have a family of $X^{*}$ over the parameter space
of its defining equations. 

Let $x_{1},...,x_{r}$ be the affine coordinate near a LCSL. Let $\mathcal{D}_{1},\cdots,\mathcal{D}_{s}$
be the Picard-Fuchs differential operators which follows from the
GKZ system characterizing the period integrals in the affine chart.
Following \cite{HKTY,HLY,HLY2}, we consider a polynomial ring $\mathbb{Q}[\theta_{1},\cdots,\theta_{r}]$
generated by $\theta_{i}:=x_{i}\frac{\partial\;}{\partial x_{i}}$
and define the \textit{indicial ideal} of the Picard-Fuchs equations,
\[
Ind(\mathcal{D}_{1},\cdots,\mathcal{D}_{s})\subset\mathbb{Q}[\theta_{1},\cdots,\theta_{r}].
\]
Indicial ideal $Ind(\mathcal{D}_{1},\cdots,\mathcal{D}_{s})$ is a
homogeneous ideal generated by initial terms of $\mathcal{D}_{i}$
and determines the \textit{indices} for the local solutions. 
\begin{prop}
There is an isomorphism 
\[
\mathbb{Q}[\theta_{1},\cdots,\theta_{r}]/Ind(\mathcal{D}_{1},\cdots,\mathcal{D}_{s})\simeq H^{0}(X,\mathbb{Q})\oplus H^{2}(X,\mathbb{Q})_{toric}\oplus H^{4}(X,\mathbb{Q}),
\]
where $H^{2}(X,\mathbb{Q})_{toric}$ is generated by the restrictions
of the ambient toric divisors. 
\end{prop}
When we normalize the top form of quotient ring, we can introduce
a pairing in the quotient ring which corresponds to the pairing in
the cohomology (of the mirror manifold). We denote this pairing for
the generators $\theta_{i}$ by 
\[
K_{ij}:=\langle\theta_{i}\theta_{j}\rangle.
\]
This represents the intersection pairing among the corresponding generators
of $H^{2}(X,\mathbb{Q})_{toric}$. 
\begin{prop}
Near the boundary point (LCSL), there is only one power series $w_{0}(x)$
representing a period integral of the mirror family, which has the
form $w_{0}(x)=\sum_{n\in\mathbb{Z}_{\geq0}^{r}}c(n)x^{n}$. All other
solutions contain logarithmic singularities, and they are given by
\begin{equation}
w_{i}^{(1)}(x):=\frac{\partial\;}{\partial\rho_{i}}w_{0}(x,\rho)\vert_{\rho=0},\;\;\;\;w^{(2)}(x):=-\frac{1}{2}\sum_{i,j}K_{ij}\frac{\partial\;}{\partial\rho_{i}}\frac{\partial\;}{\partial\rho_{j}}w_{0}(x,\rho)\vert_{\rho=0},\label{eq:FrobMethod}
\end{equation}
where $w_{0}(x,\rho)=\sum_{n}c(n+\rho)x^{n+\rho}$ with formal parameters
$\rho_{i}$. 
\end{prop}
Mirror symmetry of K3 surfaces can be summarized in the following
proposition:
\begin{prop}
$($\cite[Sect.\,2.4]{Hos}$)$ The following quadratic relation holds:
\[
2\;w_{0}(x)\left(w^{(2)}(x)+(2\pi)^{2}w_{0}(x)\right)+\sum_{i,j}K_{ij}w_{i}^{(1)}(x)w_{j}^{(1)}(x)=0.
\]

\end{prop}
Three propositions above provide a quick summary of the works \cite{HKTY,HLY,HLY2}
for the mirror symmetry expressed in the Frobenius method. It should
be noted that, while the standard Frobenius method is a well-known
technique for hypergeometric differential equations of one variables,
the Frobenius method here is a non-trivial generalization to multi-variables,
see Remark \ref{rem:Remark-GKZ-Frob} and Appendix \ref{sec:Appendix-proof}
below. 

~

~

\specialsection{\textbf{Proof of Proposition \ref{prop:Frobenius-local-sol} \label{sec:Appendix-proof}}}

Here we present the details of the proof of Proposition \ref{prop:Frobenius-local-sol}.
We also include an example which shows that a naive application of
Frobenius method results in Laurent series in general. 

\noindent  {\bf E1. Proof of \ref{prop:Frobenius-local-sol}.}
Let $\mathcal{D}_{i}$ be the Picard-Fuchs differential operators
in (\ref{eq:PF-system-o1}). Let $f_{1}(\theta),\cdots,f_{9}(\theta)$
be the homogeneous generators of the indicial ideal $Ind(\mathcal{D})$
in Proposition \ref{prop:Ind(D)}, in order. Then, for the hypergeometric
series $\omega_{0}(z,\rho)$ defined in Proposition \ref{prop:Frobenius-local-sol}
(2), it is easy to verify that 
\[
\mathcal{D}_{i}\omega_{0}(z,\rho)=z^{\rho}f_{i}(\rho)+z^{\rho}\sum_{i}z_{i}F_{i}(\rho,z),
\]
where $f_{i}(\rho)$ are the monomials $f_{i}(\theta)$ with $\theta_{i}$
replaced by $\rho_{i}$, and $F_{i}(\rho,z)$ are power series in
$z_{k}$. The indicial ideal $Ind(\mathcal{D})$ can be identified
with the ideal $I_{\rho}:=\left\langle f_{1}(\rho),\cdots,f_{9}(\rho)\right\rangle $
of the polynomial ring $\mathbb{Q}[\rho]:=\mathbb{Q}[\rho_{1},\cdots,\rho_{4}]$.
As claimed in Proposition \ref{prop:Ind-ring-Mij}, the quotient ring
$\mathbb{Q}[\rho]/I_{\rho}$ is finite dimensional with its bases
$1,\rho_{1},...,\rho_{4},\rho_{4}^{2}$. We can introduce a $\mathbb{Q}$-linear
map $\left\langle -\right\rangle :\mathbb{Q}[\rho]/I_{\rho}\to\mathbb{Q}$
by the following properties 
\[
\left\langle 1\right\rangle =\left\langle \rho_{i}\right\rangle =0,\left\langle \rho_{4}^{2}\right\rangle =d\text{ and }\left\langle h(\rho)\right\rangle =0\;(h(\rho)\in I_{\rho}),
\]
where $d\in\mathbb{Q}$ is a constant. In Proposition \ref{prop:Ind-ring-Mij},
we have introduced $M_{ij}=\left\langle \rho_{i}\rho_{j}\right\rangle $.
The next lemma follows from the above definitions:
\begin{lem}
For $g(\rho)\in\mathbb{Q}[\rho]$, we have
\[
\begin{aligned} &  & \left\langle g(\rho)\right\rangle =\frac{1}{2!}\sum_{i,j}M_{ij}\frac{\partial\,}{\partial\rho_{i}}\frac{\partial\,}{\partial\rho_{j}}g(\rho)\Big\vert_{\rho=0},\quad\left\langle \rho_{i}g(\rho)\right\rangle =\sum_{j}M_{ij}\frac{\partial\,}{\partial\rho_{j}}g(\rho)\Big\vert_{\rho=0}.\end{aligned}
\]

\end{lem}
\noindent If $g(\rho)\in I_{\rho}$, then $\left\langle g(\rho)\right\rangle =\left\langle \rho_{i}g(\rho)\right\rangle =0$
by definition. Therefore we have 
\[
\frac{1}{2!}\sum_{i,j}M_{ij}\frac{\partial\,}{\partial\rho_{i}}\frac{\partial\,}{\partial\rho_{j}}g(\rho)\Big\vert_{\rho=0}=\sum_{j}M_{ij}\frac{\partial\,}{\partial\rho_{j}}g(\rho)\Big\vert_{\rho=0}=0
\]
for $g(\rho)=f_{1}(\rho),...,f_{9}(\rho)\in I_{\rho}$. These relations
explain the form of local solutions in Proposition \ref{prop:Frobenius-local-sol}
(2), which generalizes the classical Frobenius method for hypergeometric
series of one variable. See \cite[Sect.3.3]{HLY} and an example therein
for more detailed analysis. 

The vanishing described in Proposition \ref{prop:Frobenius-local-sol}
(1) depends on the special form of the coefficient $c(n)$: 
\[
c(n)=\frac{1}{\Gamma(\frac{1}{2})^{3}}\frac{\Gamma(n_{1}+\frac{1}{2})\Gamma(n_{2}+\frac{1}{2})\Gamma(n_{3}+\frac{1}{2})}{\Pi_{_{i=1}}^{3}\Gamma(n_{4}-n_{i}+1)\cdot\Pi_{1\leq j<k\leq3}\Gamma(n_{j}+n_{k}-n_{4}+1)}.
\]
Since $c(n)$ vanishes when the $\Gamma$-functions in the denominator
have poles, it is easy to read off the necessary conditions for non-vanishing
$c(n)$ that $n_{4}-n_{i}\geq0$ and $n_{j}+n_{k}-n_{4}\geq0$. From
these conditions, we obtain $n_{4}\geq n_{1},n_{2},n_{3}\geq0$, hence
power series for $\omega_{0}(z)$. However that $\omega_{0}(z)$ is
a power series \textit{does not} guarantee that $\omega_{i}^{(1)}(z)=\frac{\partial\,}{\partial\rho_{i}}\omega_{0}(z,\rho)\Big\vert_{\rho=0}$
is a power series. In the present case (and also for GKZ hypergeometric
series near the LCSLs), we verify that this is the case directly.
We express the derivation $\frac{\partial\,}{\partial\rho_{1}}$,
for example, as
\[
\frac{\partial\,}{\partial\rho_{1}}c(n+\rho)\Big\vert_{\rho=0}=c(n)\Big\{\psi(n_{1}+\frac{1}{2})+\psi(n_{4}-n_{1}+1)-\sum_{k=2,3}\psi(n_{1}+n_{k}-n_{4}+1)\Big\},
\]
 where $\psi(s)=\frac{\Gamma(s)'}{\Gamma(s)}$. In this form we see
that there are possibilities for the poles (of $\Gamma$-functions)
in the denominator of $c(n)$ and the poles in the $\psi$-functions
cancel to result fine contributions. We write such possibilities for
each terms; e.g., we have $n_{4}-n_{1}<0$, $n_{4}-n_{i}\geq0(i=2,3)$
and $n_{j}+n_{k}-n_{4}\geq0(1\leq j<k\leq3)$ for $c(n)\psi(n_{4}-n_{1}+1)$.
From these inequalities, we obtain $n_{1},n_{2},n_{3}\geq0$ and $0\leq n_{4}<n_{1}$
and conclude that they can have non-vanishing contributions only for
$n\in\mathbb{Z}_{\geq0}^{4}$. Doing similar analysis for the other
terms, we conclude that the summation over $n\in\mathbb{Z}_{\geq0}^{4}$
in $\omega_{i}^{(1)}(z)=\frac{\partial\,}{\partial\rho_{i}}\omega_{0}(z,\rho)\Big\vert_{\rho=0}$
stays in the same range, i.e., $n\in\mathbb{Z}_{\geq0}^{4}$. Hence
we obtain the power series solution which is linear in $\log x_{1}$.
Since other cases are similar, although more involved for $\omega^{(2)}(z)$,
we omit the details. 

\noindent \textbf{E2. Laurent series from the Frobenius method. }As
is clear in the above proof, in general, hypergeometric series of
multi-variables can have some negative powers when we apply the Frobenius
method. In this respect, the content of Proposition \ref{prop:Frobenius-local-sol}
(which goes back to the observations made in \cite{HKTY,HLY}) is
that negative powers do not appear for the special boundary points
(LCSLs) coming from GKZ systems.  

To contrast the situation, we present constructions of the local solutions
of $E(3,6)$ system with the affine parameter $x=(x_{1},x_{2},x_{3},x_{4})$
defined by 
\[
\left(\begin{matrix}1 & 0 & 0 & 1 & 1 & 1\\
0 & 1 & 0 & 1 & x_{1} & x_{2}\\
0 & 0 & 1 & 1 & x_{3} & x_{4}
\end{matrix}\right).
\]
Yoshida et al \cite[Prop.1.6.1]{YoshidaEtal} studies the $E(3,6)$
system in this coordinate. Near the origin, $x=0$, there is only
one power series \cite[Prop.1.6.2]{YoshidaEtal} which we translate
to $\omega_{0}(x)=\sum_{n\in\mathbb{Z}_{\geq0}^{4}}c(n)x^{n}$ with
\[
c(n)=\frac{\Gamma(n_{1}+n_{3}+\frac{1}{2})\Gamma(n_{2}+n_{4}+\frac{1}{2})\Gamma(n_{1}+n_{2}+\frac{1}{2})\Gamma(n_{3}+n_{4}+\frac{1}{2})}{\Gamma(n_{1}+n_{2}+n_{3}+n_{4}+\frac{3}{2})\Gamma(n_{1}+1)\Gamma(n_{2}+1)\Gamma(n_{3}+1)\Gamma(n_{4}+1)}.
\]
By explicit calculations, we can verify that all other solutions contains
logarithms, $\log x_{i}$ and $(\log x_{i})(\log x_{j})\,(1\leq i,j\leq4)$.
For example, the solutions linear in $\log x_{1}$ is given by the
naive application of the Frobenius method  $\frac{\partial\,}{\partial\rho_{1}}\omega_{0}(x,\rho)\Big\vert_{\rho=0}.$
However this becomes a Laurent series as we can deduce from 
\[
\frac{\partial\,}{\partial\rho_{1}}c(n_{1},0,0,0)=\frac{\Gamma(n_{1}+\frac{1}{2})^{2}\Gamma(\frac{1}{2})^{2}}{\Gamma(n_{1}+\frac{3}{2})\Gamma(n_{1}+1)}\big\{2\psi(n_{1}+\frac{1}{2})-\psi(n_{1}+\frac{3}{2})-\psi(n_{1}+1)\big\}.
\]
By explicit calculation, we verify that the Laurent series 
\[
\frac{\partial\,}{\partial\rho_{1}}\omega_{0}(x,\rho)\Big\vert_{\rho=0}=\sum_{n\in\mathbb{Z}^{4}}\frac{\partial\,}{\partial\rho_{1}}c(n+\rho)\Big\vert_{\rho=0}x^{n}+\omega_{0}(x)\log x_{1}
\]
satisfies the $E(3,6)$ system given in Yoshida et al \cite[Prop.1.6.1]{YoshidaEtal}.
Similar calculations works for the other logarithmic solutions as
well. 

~ 

~

\specialsection{\textbf{Picard-Fuchs operators for $\widetilde{\mathcal{X}}^{+}\to\mathcal{X}$
\label{sec:App.PF-equation-Bo1}}}

There are three LCSLs $o_{i}^{+}(i=1,2,3)$ in the flipped resolution
$\widetilde{\mathcal{X}}^{+}\to\mathcal{X}$. We introduce the local
affine coordinates $\tilde{z}_{k}:=z_{k}(o_{1}^{+}),\,\tilde{z}_{k}':=z_{k}(o_{2}^{+})$
and $\tilde{z}_{k}'':=z_{k}(o_{3}^{+})$ whose origins are the LCSLs
$o_{i}^{+}$. Here, following Lemma 3.3 and Definition 3.5 of \cite{HLTYpartI},
we describe the coordinate $\tilde{z}_{k}:=z_{k}(o_{1}^{+})$ and
also Picard-Fuchs differential operators $\tilde{\mathcal{D}}$ explicitly.
Following the notation in \cite{HLTYpartI}, we first note that the
cone $(\sigma_{1}^{(2)})^{\vee}\cap L$ is generated by 


\begin{equation}
\begin{matrix}\tilde{\ell}^{(1)}=(-1,\;\;\,0,\;\;\,0,\;\;\,1,\;\;\,0,\;\;\,0,\;\;\,0,\;\;\,1,-1)=\ell^{(1)},\qquad\;\;\,\,\\
\tilde{\ell}^{(2)}=(-1,\,\,\,\,0,\;\;\,0,\;\;\,0,\;\;\,1,-1,\;\;\,1,\;\;\,0,\;\;\,0)=\ell^{(1)}+\ell^{(4)},\\
\tilde{\ell}^{(3)}=(\;\;\,1,-1,\;\;\,0,\;\;\;0,-1,\;\;\,1,\;\;\,0,-1,\;\;\,1)=\ell^{(2)}-\ell^{(1)},\\
\tilde{\ell}^{(4)}=(\;\;\,1,\;\;0,-1,\;-1,\;\;0,\;\;\,1,-1,\;\;\,0,\;\;\,1)=\ell^{(3)}-\ell^{(1)}.
\end{matrix}\label{eq:tilde-ell-by-ell}
\end{equation}
The coordinates $\tilde{z}_{k}:=z_{k}(o_{1}^{+})$ follow from these
by $\tilde{z}_{i}:=\mathtt{a}^{\tilde{\ell}^{(i)}}$, i.e., 
\[
\tilde{z}_{1}=-\frac{a_{1}c_{1}}{a_{0}c_{2}},\;\tilde{z}_{2}=-\frac{a_{2}b_{2}}{a_{0}b_{1}},\;\tilde{z}_{3}=\frac{a_{0}b_{1}c_{2}}{a_{2}b_{0}c_{1}},\;\tilde{z}_{4}=\frac{a_{0}b_{1}c_{2}}{a_{1}b_{2}c_{0}}.
\]
 It turns out that a complete set of differential operators $\tilde{\mathcal{D}}_{\ell}$
are given by the following $\tilde{\ell}$'s (cf. \cite[Appendix C]{HLTYpartI}):
\[
\begin{matrix}\tilde{\ell}^{(1)},\;\tilde{\ell}^{(2)},\;\tilde{\ell}^{(2)}+\tilde{\ell}^{(3)},\;\tilde{\ell}^{(2)}+\tilde{\ell}^{(4)},\;\tilde{\ell}^{(1)}+\tilde{\ell}^{(4)},\;\tilde{\ell}^{(1)}+\tilde{\ell}^{(3)}\\
\tilde{\ell}^{(1)}+\tilde{\ell}^{(2)}+\tilde{\ell}^{(3)},\;\tilde{\ell}^{(1)}+\tilde{\ell}^{(2)}+\tilde{\ell}^{(4)},\;\tilde{\ell}^{(1)}+\tilde{\ell}^{(2)}+\tilde{\ell}^{(3)}+\tilde{\ell}^{(4)}
\end{matrix}
\]
Setting $\tilde{\theta}_{i}:=\tilde{z}_{i}\frac{\partial\;}{\partial\tilde{z}_{i}}$,
the operators take the following forms:
\[
\begin{matrix}\tilde{\mathcal{D}}_{1}=(\tilde{\theta}_{1}-\tilde{\theta}_{3})(\tilde{\theta}_{1}-\tilde{\theta}_{4})+\tilde{z}_{1}(\tilde{\theta}_{1}-\tilde{\theta}_{3}-\tilde{\theta}_{4})(\tilde{\theta}_{1}+\tilde{\theta}_{2}-\tilde{\theta}_{3}-\tilde{\theta}_{4}+\frac{1}{2}),\\
\tilde{\mathcal{D}}_{2}=(\tilde{\theta}_{2}-\tilde{\theta}_{3})(\tilde{\theta}_{2}-\tilde{\theta}_{4})+\tilde{z}_{2}(\tilde{\theta}_{2}-\tilde{\theta}_{3}-\tilde{\theta}_{4})(\tilde{\theta}_{1}+\tilde{\theta}_{2}-\tilde{\theta}_{3}-\tilde{\theta}_{4}+\frac{1}{2}),\\
\tilde{\mathcal{D}}_{3}=(\tilde{\theta}_{1}-\tilde{\theta}_{4})(\tilde{\theta}_{2}-\tilde{\theta}_{3}-\tilde{\theta}_{4})+\tilde{z}_{1}\tilde{z}_{3}(\tilde{\theta}_{2}-\tilde{\theta}_{3})(\tilde{\theta}_{3}+\frac{1}{2}),\qquad\qquad\quad\;\\
\tilde{\mathcal{D}}_{4}=(\tilde{\theta}_{1}-\tilde{\theta}_{3})(\tilde{\theta}_{2}-\tilde{\theta}_{3}-\tilde{\theta}_{4})+\tilde{z}_{1}\tilde{z}_{4}(\tilde{\theta}_{2}-\tilde{\theta}_{4})(\tilde{\theta}_{4}+\frac{1}{2}),\qquad\qquad\quad\;\\
\tilde{\mathcal{D}}_{5}=(\tilde{\theta}_{2}-\tilde{\theta}_{4})(\tilde{\theta}_{1}-\tilde{\theta}_{3}-\tilde{\theta}_{4})+\tilde{z}_{2}\tilde{z}_{3}(\tilde{\theta}_{1}-\tilde{\theta}_{3})(\tilde{\theta}_{3}+\frac{1}{2}),\qquad\qquad\quad\;\\
\tilde{\mathcal{D}}_{6}=(\tilde{\theta}_{2}-\tilde{\theta}_{3})(\tilde{\theta}_{1}-\tilde{\theta}_{3}-\tilde{\theta}_{4})+\tilde{z}_{2}\tilde{z}_{4}(\tilde{\theta}_{1}-\tilde{\theta}_{4})(\tilde{\theta}_{4}+\frac{1}{2}),\qquad\qquad\quad\;\\
\tilde{\mathcal{D}}_{7}=(\tilde{\theta}_{1}-\tilde{\theta}_{4})(\tilde{\theta}_{2}-\tilde{\theta}_{4})-\tilde{z}_{1}\tilde{z}_{2}\tilde{z}_{3}(\tilde{\theta}_{3}+\frac{1}{2})(\tilde{\theta}_{1}+\tilde{\theta}_{2}-\tilde{\theta}_{3}-\tilde{\theta}_{4}+\frac{1}{2}),\;\\
\tilde{\mathcal{D}}_{8}=(\tilde{\theta}_{1}-\tilde{\theta}_{3})(\tilde{\theta}_{2}-\tilde{\theta}_{3})-\tilde{z}_{1}\tilde{z}_{2}\tilde{z}_{4}(\tilde{\theta}_{4}+\frac{1}{2})(\tilde{\theta}_{1}+\tilde{\theta}_{2}-\tilde{\theta}_{3}-\tilde{\theta}_{4}+\frac{1}{2}),\;\\
\tilde{\mathcal{D}}_{9}=(\tilde{\theta}_{1}-\tilde{\theta}_{3}-\tilde{\theta}_{4})(\tilde{\theta}_{2}-\tilde{\theta}_{3}-\tilde{\theta}_{4})-\tilde{z}_{1}\tilde{z}_{2}\tilde{z}_{3}\tilde{z}_{4}(\tilde{\theta}_{3}+\frac{1}{2})(\tilde{\theta}_{4}+\frac{1}{2}).\quad\;
\end{matrix}
\]
The radical $\sqrt{dis}$ of the discriminant is given by

\[
\begin{matrix}\begin{aligned}\tilde{z}_{1}\tilde{z}_{2}\tilde{z}_{3}\tilde{z}_{4}\times\prod_{i=1}^{2}(1+\tilde{z}_{i})\times\prod_{{i=1,2\atop j=3,4}}(1+\tilde{z}_{i}\tilde{z}_{j})\times\prod_{j=3,4}(1-\tilde{z}_{1}\tilde{z}_{2}\tilde{z}_{j})\qquad\\
\;\;\times(1-\tilde{z}_{1}\tilde{z}_{2}\tilde{z}_{3}\tilde{z}_{4})\times\big(1-\tilde{z}_{1}\tilde{z}_{2}(\tilde{z}_{3}+\tilde{z}_{4}+\tilde{z}_{3}\tilde{z}_{4})-(\tilde{z}_{1}+\tilde{z}_{2})\tilde{z}_{1}\tilde{z}_{2}\tilde{z}_{3}\tilde{z}_{4}\big).
\end{aligned}
\end{matrix}
\]
The pairing $\tilde{M}_{ij}:=\langle\tilde{\theta}_{i}\tilde{\theta}_{j}\rangle$
from the quotient ring $\mathbb{Q}[\tilde{\theta}_{1},\cdots,\tilde{\theta}_{4}]/Ind(\tilde{D})$
is determined as 
\[
\left(\tilde{M}_{ij}\right)=2\times\left(\begin{matrix}1 & 2 & 1 & 1\\
2 & 1 & 1 & 1\\
1 & 1 & 0 & 1\\
1 & 1 & 1 & 0
\end{matrix}\right)=2\times\,^{t}\tilde{T}\left(\begin{matrix}0 & 1 & 0 & 0\\
1 & 0 & 0 & 0\\
0 & 0 & -1 & 0\\
0 & 0 & 0 & -1
\end{matrix}\right)\tilde{T}
\]
with $\tilde{T}=\left(\begin{smallmatrix}1 & 1 & 0 & 1\\
1 & 1 & 1 & 0\\
0 & 1 & 0 & 0\\
1 & 0 & 0 & 0
\end{smallmatrix}\right)$ and fixing the normalization $\tilde{d}$ by $2$ (cf. Proposition
\ref{prop:Ind-ring-Mij}). As in Proposition \ref{prop:Frobenius-local-sol},
we have local solutions around $\tilde{z}_{i}=0\,(i=1,..,4)$ via
the Frobenius method using the above $\tilde{M}_{ij}$. The corresponding
quadratic relation also holds with $M_{ij}$ replaced by $\tilde{M}_{ij}$
in Proposition \ref{prop:period-relation-Mij}

~As a unique (up to constant) solution, we obtain 
\[
\omega_{0}(\tilde{z})=\sum_{n_{1},n_{2},n_{3},n_{4}\geq0}\tilde{c}(n_{1},n_{2},n_{3},n_{4})\tilde{z}_{1}^{n_{1}}\tilde{z}_{2}^{n_{2}}\tilde{z}_{3}^{n_{3}}\tilde{z}_{4}^{n_{4}}
\]
with $\tilde{c}(n)=\tilde{c}(n_{1},n_{2},n_{3},n_{4})$ given by 
\[
\tilde{c}(n):=\frac{1}{\Gamma(\frac{1}{2})^{3}}\frac{\Gamma(n_{1}+n_{2}-n_{3}-n_{4}+\frac{1}{2})\Gamma(n_{3}+\frac{1}{2})\Gamma(n_{4}+\frac{1}{2})}{\Pi_{i=1,2}\Pi_{j=3,4}\Gamma(n_{i}-n_{j}+1)\cdot\Pi_{i=1,2}\Gamma(n_{3}+n_{4}-n_{i}+1)}.
\]
The following proposition follows from constructing the integral generators
of the semi-groups $(\sigma_{i}^{(2)})^{\vee}\cap L$ $(i=1,2,3)$
described in Lemma 3.3 of \cite{HLTYpartI}.
\begin{prop}
The other two affine coordinates $\tilde{z}_{k}':=z_{k}(o_{2}^{+})$
and $\tilde{z}_{k}'':=z_{k}(o_{3}^{+})$ are related to $\tilde{z}_{k}$
by 
\begin{equation}
\begin{aligned}\tilde{z}_{1}' & =\tilde{z}_{1}\tilde{z}_{4}, &  &  & \tilde{z}_{2}' & =\tilde{z}_{2}\tilde{z}_{4}, &  &  & \tilde{z}_{3}' & =\frac{\tilde{z}_{3}}{\tilde{z}_{4}}, &  &  & \tilde{z}_{4}' & =\frac{1}{\tilde{z_{4}}},\\
\tilde{z}_{1}'' & =\tilde{z}_{1}\tilde{z}_{3}, &  &  & \tilde{z}_{2}'' & =\tilde{z}_{2}\tilde{z}_{3}, &  &  & \tilde{z}_{3}'' & =\frac{\tilde{z}_{4}}{\tilde{z}_{3}}, &  &  & \tilde{z}_{4}'' & =\frac{1}{\tilde{z_{3}}}.
\end{aligned}
\label{eq:zo2plus-zo3plus}
\end{equation}

\end{prop}
By using these relations, it is straightforward to express the Picard-Fuchs
operators $\widetilde{\mathcal{D}}_{i}(i=1,...,9)$ in the coordinates
$\tilde{z}_{k}':=z_{k}(o_{2}^{+})$ and also $\tilde{z}_{k}'':=z_{k}(o_{3}^{+})$.
We leave it to the reader to see that the set of operators $\left\{ \widetilde{\mathcal{D}}_{i}\right\} $
has the same form for these three coordinates. 

\vspace{0.5cm}

{\footnotesize{}Department of Mathematics, Gakushuin University, }{\footnotesize \par}

{\footnotesize{}Mejiro, Toshima-ku, Tokyo 171-8588, Japan }{\footnotesize \par}

{\footnotesize{}e-mail: hosono@math.gakushuin.ac.jp}{\footnotesize \par}

~

{\footnotesize{}Department of Mathematics, Brandeis University, }{\footnotesize \par}

{\footnotesize{}Waltham MA 02454, U.S.A. }{\footnotesize \par}

{\footnotesize{}e-mail: lian@brandeis.edu}{\footnotesize \par}

~

{\footnotesize{}Department of Mathematics, Harvard University, }{\footnotesize \par}

{\footnotesize{}Cambridge MA 02138, U.S.A. }{\footnotesize \par}

{\footnotesize{}e-mail: yau@math.harvard.edu}{\footnotesize \par}

~

~

~
\end{document}